\documentclass[a4paper,11pt, twoside]{amsart}

\usepackage{amsmath,amsfonts,amsthm,amsopn,xcolor,amssymb,enumitem}
\usepackage{palatino}
\usepackage[colorinlistoftodos]{todonotes}

\usepackage[colorlinks=true,urlcolor=blue,citecolor=blue,linkcolor=blue,linktocpage,pdfpagelabels,bookmarksnumbered,bookmarksopen]{hyperref}
\hypersetup{urlcolor=blue, citecolor=blue, linkcolor=blue}

\usepackage{bm}
\usepackage{mathtools}

\usepackage[italian,english]{babel}
\usepackage[left=2.61cm,right=2.61cm,top=2.72cm,bottom=2.72cm]{geometry}

\newcommand{\eps}{\varepsilon}
\newcommand{\C}{\mathbb{C}}
\newcommand{\R}{\mathbb{R}}

\newcommand{\RN}{{\mathbb{R}^N}}
\newcommand{\RD}{{\mathbb{R}^2}}
\newcommand{\RT}{{\mathbb{R}^3}}

\newcommand{\de}{\partial}

\renewcommand{\le}{\leqslant}
\renewcommand{\ge}{\geqslant}
\renewcommand{\a }{\alpha }

\renewcommand{\d }{\delta }
\newcommand{\vfi}{\varphi}
\newcommand{\g }{\gamma }

\renewcommand{\l }{\lambda}

\newcommand{\n }{\nabla }
\newcommand{\s }{\sigma }
\renewcommand{\t}{\theta}

\newcommand{\G}{\Gamma}

\renewcommand{\H}{H^1(\RN)}
\newcommand{\Hr}{H^1_{\rm rad}(\RN)}
\newcommand{\Ha}{H^1_\a(\RN)}

\newcommand{\Hal}{H^1_{\a ,\l}}
\newcommand{\Har}{H^1_{\a ,{\rm rad}}(\RN)}

\newcommand{\cg}{\mathcal{G}}

\newcommand{\N}{\mathbb{N}}

\renewcommand{\C}{\mathbb{C}}
\newcommand{\re}{\operatorname{Re}}

\newcommand{\ird }{\int_{\RD}}
\newcommand{\irn }{\int_{\RN}}

\def\bbm[#1]{\mbox{\boldmath $#1$}}
\newcommand{\beq }{\begin{equation}}
\newcommand{\eeq }{\end{equation}}

\newcommand{\weakto}{\rightharpoonup}

\newcommand{\dis}{\displaystyle}

\newcommand{\gl}{\mathcal{G}_{\lambda}}
\newcommand{\gr}{\mathcal{G}_{\lambda, {\rm reg}}}
\newcommand{\gs}{\mathcal{G}_{{\rm sing}}}

\newcommand{\im}{\operatorname{Im}}

\newcommand{\ef}{\eqref}
\newcommand{\lr}[1]{\langle#1\rangle}

\newcommand{\red}[1]{{\color{black}#1}}

\makeatletter
\providecommand\@dotsep{5}
\def\listtodoname{List of Todos}
\def\listoftodos{\@starttoc{tdo}\listtodoname}
\makeatother

\newtheorem{theorem}{Theorem}[section]
\newtheorem{lemma}[theorem]{Lemma}
\newtheorem{remark}[theorem]{Remark}
\newtheorem{proposition}[theorem]{Proposition}

\numberwithin{equation}{section}

\allowdisplaybreaks

\title
[
Nonlinear scalar field equation with point interaction
]
{
Nonlinear scalar field equation with point  interaction
}

\author[A. Pomponio]{Alessio Pomponio}
\author[T. Watanabe]{Tatsuya Watanabe}

\address[A. Pomponio]{\newline\indent
Dipartimento di Meccanica, Matematica e Management
\newline\indent 
Politecnico di Bari
\newline\indent
Via Orabona 4,  70125  Bari, Italy}
\email{alessio.pomponio@poliba.it}
\address[T. Watanabe]{\newline\indent 
Department of Mathematics, 
\newline\indent 
Faculty of Science, Kyoto Sangyo University,
\newline\indent
Motoyama, Kamigamo, Kita-ku, Kyoto-City, 603-8555, Japan}
\email{tatsuw@cc.kyoto-su.ac.jp}
\thanks{}

\date{}

\begin{document}

\begin{abstract}
This paper is devoted to the study of the nonlinear scalar field equation 
with a point interaction at the origin in dimensions two and three.
\red{More precisely, in dimension three and in the repulsive case, 
by applying the mountain pass theorem and the technique of adding one dimensional space, 
we prove the existence of a nontrivial singular solution for a wide class of nonlinearities. 
In dimension $2$, or in dimension $3$ and in the attractive case,
the existence result requires the Ambrosetti-Rabinowitz condition and the arguments become simpler.}
We also establish the Pohozaev identity by proving a pointwise 
estimate of the gradient near the origin.
Some qualitative properties of nontrivial solutions are also given.
\end{abstract}

\maketitle

\section{Introduction}
In this paper, we study the following 
nonlinear elliptic problem with $\delta$-interaction
\begin{equation} \label{eq:delta} 
\begin{cases}
\red{-\Delta_\a u }= g(u) & \hbox{  in }\RN, \\
u(x)\to 0& \hbox{ as }|x|\to +\infty,
\end{cases}
\end{equation}
\red{where $N=2, 3$ and $-\Delta_\a$ is the Schr\"odinger operator with a point interaction at the origin 
parametrized by the interaction strength $\a \in \R$ (see \ef{del_alpha} below for the precise definition). 
This operator can be considered as the two or three dimensional version of the Schr\"odinger operator 
$-\frac{d^2}{dx^2} + \a \d_0$ in one dimension where $\delta_0$ is the delta function supported at the origin.
Notably, the scalar field equation $-\Delta u = g(u)$ 
can be regarded as a limit equation of \ef{eq:delta} as $\alpha \to +\infty$.}
Equation \ef{eq:delta} can be obtained by 
considering the standing wave $\psi (t,x)= e^{i \omega t} u(x)$
for the nonlinear Schr\"odinger equation (NLS)
\begin{equation} \label{NLS}
\red{i \psi_t + \Delta_\a \psi  + h( |\psi|) \frac{\psi}{|\psi|} =0,}
\end{equation}
provided that $g(s)=h(s)- \omega s$ and $\omega \in \R$.
NLS with point interaction has been recently proposed 
as an effective model for a Bose-Einstein Condensate (BEC) 
in the presence of defects or impurities.
See \cite{SM, SCMS} for the physical background.
In the 1D case,
there has been a lot of works for \ef{eq:delta} and \ef{NLS},
such as the existence of a ground state solution
and the (in)stability of standing waves;
we refer to \cite{BC, FJ, FOO, KO} and references therein.
On the other hand, the higher dimensional case is less studied.
2D problem has been studied 
for the pure power case $h(s) = |s|^{p-2}s$ in \cite{ABCT2, FGI},
while 3D problem with $h(s) = |s|^{p-2}s$
has been investigated in \cite{ABCT3}.
See also \cite{Ten} for a survey. 
Concerning with time-dependent problems in higher dimensional case,
we refer to \cite{CFN1, CFN2, FN, GMS} and references therein.
In \cite{OP}, instead, existence and asymptotic behavior is considered 
for a system of coupled nonlinear Schr\"odinger equations with point interaction.

The purpose of this paper is to consider \ef{eq:delta} for general $g$, 
in the spirit of \cite{BL}, prove the existence of a nontrivial solution 
and investigate qualitative properties of any nontrivial solutions of \ef{eq:delta}.

Equation \ef{eq:delta} is formal since the delta interaction is not
a small perturbation of $-\Delta$ in general.
A rigorous formulation is given through the self-adjoint extension
of the operator $-\Delta|_{C_0^{\infty}(\R^N \setminus \{0 \})}$.
Then it is known that there exists a family $\{ - \Delta_{\alpha} \}_{\alpha \in \R}$
of self-adjoint operators which realize all point perturbations of $-\Delta$;
see \cite{AGH, AGHH1, AGHH2, AH}.
As a consequence, the domain of $-\Delta_{\alpha}$ is given by 
\begin{align*}
D(-\Delta_\a) &:=
\big\{ u\in L^2(\RN): \text{there exist} \ q=q(u) \in \C \text{ and }\l >0\text{ s.t. } \\
&\qquad \qquad
\phi_\l:=u-q(u)\cg_\l\in H^2(\RN) \text{ and }\phi_\l(0)=(\a +\xi_\l)q(u) \big\},
\end{align*}
where $\cg_\l$ is the Green's function of $-\Delta +\l$ on $\R^N$, 
\begin{equation}\label{xi}
\xi_\l:=
\begin{cases}
\dis \frac{\sqrt{\l}}{4\pi} &\ (N=3), \\ \smallskip
\dis \frac{\log\left(\frac{\sqrt{\l}}{2}\right)+\g}{2\pi} &\  (N=2),
\end{cases}
\end{equation}
and $\g $ is the Euler-Mascheroni constant.
Moreover the action is defined by
\begin{equation} \label{del_alpha}
-\Delta_\a u:= -\Delta \phi_\l - q(u) \l \cg_\l, \quad \text{for all } u\in D(-\Delta_\a).
\end{equation}
It is also known that $\sigma_{ess}(-\Delta_{\alpha})=[0,\infty)$.
Moreover when $N=2$, or $N=3$ and $\alpha<0$, 
$-\Delta_{\alpha}$ has exactly one negative eigenvalue $-\omega_{\alpha}$
which is given by
\begin{equation}\label{omega-al}
\omega_\a:=\begin{cases}
4e^{-4\pi \a -2\g } &\ \text{for }N=2 \ \text{and} \ \a \in \R, \\
(4\pi \a)^2 &\  \text{for }N=3 \ \text{and} \ \a <0. 
\end{cases}
\end{equation}
For convenience, let us put 
$$
\omega_{\alpha}:=0\qquad\text{ when } N=3\text{ and }\alpha \ge 0.
$$
By the definition of $\xi_\l$ in \ef{xi}, 
we find that $\alpha + \xi_\l >0$ for any $\l > \omega_{\alpha}$.
Under these preparations, 
the rigorous version of \ef{eq:delta} can be formulated as follows:
\begin{equation} \label{eq2} 
\begin{cases}
-\Delta \phi_\l - \l q(u) \gl = g(u) & \hbox{  in } L^2(\RN), \\
u\in D(-\Delta_\a).
\end{cases}
\end{equation}

The function $u \in D(-\Delta_{\alpha})$ consists of a \textit{regular part} $\phi_\l$,
on which $-\Delta_{\alpha}$ acts as the standard Laplacian, 
and a \textit{singular part} $q(u) \gl$,
on which $-\Delta_{\alpha}$ acts as the multiplication by $-\l$.
These two components are coupled by the boundary condition 
$\phi_\l(0)= (\alpha + \xi_\l)q(u)$.
The strength $q=q(u)$ is called a \textit{charge} of $u$.
In particular, we have that 
\[
\lr{-\Delta_\a u,u}
= \|\n \phi_\l\|_2^2+\l \|\phi_\l \|_2^2 - \lambda \| u\|_2^2
+(\a +\xi_\l)|q(u)|^2.
\]
As observed in \cite[Remark 2.1]{ABCT2}, $\l$ is a free parameter
and it does not affect the definition of $-\Delta_{\alpha}$ 
nor the charge $q(u)$; see also \ef{eq:2.3} below.
It is also remarkable that $-\Delta|_{C_0^{\infty}(\R^N \setminus \{0 \})}$
is essentially self-adjoint for $N \ge 4$ and $\gl \in L^2(\R^N)$ only if $1 \le N \le 3$,
which means that $\delta$-interaction makes sense only when $1 \le N \le 3$.

As mentioned above, the existence of a ground state solution of \ef{eq2}
and its qualitative properties  
for the case $g(s)= - \omega s +|s|^{p-2}s$ 
have been established in \cite{ABCT2, ABCT3, FGI}.
Their proof heavily rely on the homogeneity of the nonlinear term,
which enables us to characterize the ground state solution 
as a minimizer of the Nehari manifold.
Our purpose is to extend their existence results for a wide class of nonlinearities.
Especially we aim to obtain the existence of a nontrivial singular solution
\textit{without using the Nehari manifold}.
We also establish the Pohozaev identity for \ef{eq2}, 
which is independently interesting and useful for further investigations.

\medskip
To state our main theorems, let us define 
the energy space associated with \ef{eq2} by 
\begin{equation*}
\Ha\!:=\!
\left\{ u\!\in\! L^2(\RN)\!:\! \text{there exist } 
q=q(u)\! \in \!\C \text{ and }\l >0\text{ s.t. } 
\phi_\l:=u-q(u) \cg_\l\!\in \! H^1(\RN) \right\}\!.
\end{equation*}
We remark that even if we work on this low regularity space,
$q(u)$ is independent of $\l$ and uniquely determined,
as shown in Lemma \ref{lem:2.3} below.
Therefore, in the definition of $\Ha$, 
we do not stress the dependence of $q(u)$ with respect to $\l$.

For any $\l>\omega_\a$, we define the related quadratic form by
\begin{equation*}
\lr{(-\Delta_\a+\l )u,u} := \|\n \phi_\l\|_2^2+\l \|\phi_\l \|_2^2
+(\a +\xi_\l)|q(u)|^2,
\end{equation*}
for $u=\phi_\l+q(u)\cg_\l \in \Ha$.
Here $\langle \cdot, \cdot \rangle$ denotes the standard $L^2$-inner product.
We also put
\begin{equation*} \label{norm}
\|u\|_{\Hal}^2:=\lr{(-\Delta_\a+\l )u,u}
= \|\n \phi_\l\|_2^2+\l \|\phi_\l \|_2^2 
+(\a +\xi_\l)|q(u)|^2.
\end{equation*}
Clearly if $q(u)=0$, then $\|u\|_{\Hal}$ coincides with 
the norm $\|  u \|_{H^1}$.
Moreover it also holds that
\begin{equation} \label{equiv}
\| u\|_{H^1_{\alpha,\lambda_1}} \sim \| u \|_{H^1_{\alpha,\l_2}}
\quad \text{for} \ \ \omega_{\alpha} <\l_1 < \l_2.
\end{equation}
See \cite{FGI} for details.

On the nonlinearity $g$, we require that
\begin{enumerate} [label=(g\arabic{*}),ref=g\arabic{*}]
\item \label{g1} $g\in C([0,\infty),\R)$;
\item \label{g2} there exists $\omega \in (\omega_\a, +\infty)$ such that
$$ 
-\infty <\liminf_{s\to 0^+} \frac{g(s)}{s} 
\le \limsup_{s \to  0^+} \frac{g(s)}{s} =-\omega;
$$
\item \label{g3} it holds that
\begin{equation*} 
\displaystyle -\infty < \lim_{s \to +\infty}
\frac{g(s)}{s^{p-1}} \le 0 \quad\text{for some }
\begin{cases}
2<p<3 &\  (N=3), \\
p>2 &\  (N=2);
\end{cases}
\end{equation*}
\item \label{g4} there exists $\zeta>0$ such that $G(\zeta) >0$,
where $G(s)=\int_0^{s} g(\tau) \,d\tau $. 
\end{enumerate}
We extend $g$ and $G$ to the complex plane by setting, by an abuse of notation,
\[
g(u)= g(|u|) \frac{u}{|u|} \quad \text{and} \quad G(u)=G(|u|),
\qquad \text{for} \ u \in \C, u \ne 0.
\]
Then $g$ is odd and $G$ is even on $\R$.
Moreover $\im \{ g(u)\bar{u} \}=0$ and
$g$ is gauge invariant, i.e. $g(e^{i \theta} s) = e^{i \theta} g(s)$,
for $\theta \in \R$ and $s \in \R$.
We emphasize that we can treat a wide class of nonlinearities,
such as \textit{double power} nonlinearity 
$g(s) = -\omega s - |s|^{p_1-2}s + |s|^{p_2-2}s$,
$g(s) = -\omega s + \mu |s|^{p_1-2}s \pm |s|^{p_2-2}s$, 
with $2<p_1<p_2$ and $p_2<3$, if $N=3$, and for suitable $\mu>0$
and superlinear nonlinearity $g(s)= - \omega s + |s|^{p-2}s \log ( |s| +1)$, 
with $2<p$ and $p<3$, if $N=3$. 

\red{For our existence results, however, as explained in Remark \ref{rem:5.7}, 
in the case $N=3$, $\alpha \le 0$ or $N=2$, in place of \eqref{g4},
we need to require a stronger assumption, namely, 
the \textit{Ambrosetti-Rabinowitz} growth condition:
\begin{enumerate} [label=(g\arabic{*}),ref=g\arabic{*}]
\setcounter{enumi}{4}
\item \label{g5} 
there exists $\beta > 2$ such that for $h(s) := g(s) + \omega s$, it holds that
\[
0 < \beta H(s) \le h(s)s \quad \text{for all} \ s >0.
\]
\end{enumerate}
By the extension of $g$, it also follows that
$0< \beta H(u) \le h(u) \bar{u}$,  for any  $u \in \C$, $u \ne 0$.
Condition \ef{g5} implies that $H(s)$ has a super-quadratic growth at infinity.
Typical examples of $h(s)$ satisfying \ef{g5} except from the power nonlinearity
are given by $h(s)= |s|^{p-2}s \log (|s|+e)$ or
$h(s)= \frac{|s|^{p-2}s}{1+|s|^{p-q}}$
for $2<p,q$ and $p<3$ if $N=3$. (See also \cite[Section 2.5]{BadSer}.)
}

We define the energy functional $I: \Ha \to \R$ by
\begin{align*}
I(u)&:= \frac 12\lr{-\Delta_\a u,u} - \irn G(u) \,dx
\\
&\ = \frac 12\|\n \phi_\l\|_2^2+\frac \l 2\|\phi_\l \|_2^2-\frac \l 2\|u\|_2^2
+\frac 12(\a +\xi_\l)|q(u)|^2 - \irn G(u) \,dx,
\end{align*}
for $\lambda >0$ and $u=\phi_\l+q(u)\cg_\l\in \Ha$.
It is important to note that the value of $I$ is independent of 
the choice of $\l$.
We will see in Proposition \ref{prop:2.8} below that
any solution $u$ of \ef{eq2} is a critical point of $I$.
On the other hand,
we will also show in Proposition \ref{prop:2.8}
that any critical point $u= \phi_\l + q(u) \gl$ of $I$ is a weak solution of 
\begin{equation} \label{eq}
-\Delta \phi_\l - \l q(u) \gl = g(u) \quad \text{in} \ \R^N,
\end{equation}
that is, $u$ satisfies
\[
\re\left\{\langle \nabla \phi_\l, \nabla \psi_\l \rangle 
+ \l \langle \phi_\l, \psi_\l \rangle - \l \langle u,v \rangle 
+(\alpha + \xi_\l ) q(u) \overline{q(v)} \right\}
=\re \irn g(u) \bar{v} \,dx, 
\]
for all $ v \in \Ha$. 
Here $\re$ denotes the real part.
Moreover by Proposition \ref{prop:5.1} below, 
any weak solution of \ef{eq} satisfies the boundary condition
\begin{equation} \label{bdry}
\phi_\l(0)= (\alpha + \xi_\l) q(u).
\end{equation}
Thus by the definition of $D(-\Delta_{\alpha})$, 
a solution $u= \phi_\l + q(u) \gl$ of \ef{eq} satisfying \ef{bdry}
is actually a solution of the original problem \ef{eq2}
\textit{only if} $\phi_\l$ belongs to $H^2(\R^N)$.
Therefore, a critical point of $I$ is \textit{not} a solution of the original problem \ef{eq2},
in case $H^2$-regularity of weak solutions cannot be established.
As we will see in the following  Remark \ref{rem:1.2}, 
this strange phenomenon may occur in three dimensions. 
We also mention that the constant $\alpha$ in \ef{eq:delta}
does not appear directly in \ef{eq}
but is included in the boundary condition \ef{bdry}.

In this setting, first we study the relation between weak solutions of \eqref{eq}, 
the boundary condition \eqref{bdry} and solutions of \eqref{eq2}. 
To this aim, we have to analyse the regularity of solutions 
discovering that the situation is more delicate in dimension $N=3$
(see Remark \ref{rem:1.2} below). 
These regularity results will be also useful to establish a Pohozaev type identity. 
More precisely, we are able to obtain the following results.

\begin{theorem} \label{main2}
Assume \eqref{g1}--\eqref{g3}
and let $u=\phi_\l+q(u)\cg_\l$ be any nontrivial weak solution of \ef{eq}.
Then $\phi_\l$ is continuous at the origin,
$\phi_\l \not\equiv 0$ for any $\l>0$,
$q(u)$ is non-negative up to phase shift
and the boundary condition \ef{bdry} holds. 
Moreover $u$ satisfies the Pohozaev identity:
\begin{equation} \label{eq:5.6}
\begin{aligned}
0 &= \frac{N-2}{2}\left\|\nabla \phi_\lambda\right\|_2^2
+\frac{(N-2) \lambda}{2} \left(\left\|\phi_\l\right\|_2^2-\| u\|_2^2\right)
- \l  \| \gl \|_2^2 |q(u)|^2 \\
&\quad +(N-2)(\alpha + \xi_\l) |q(u)|^2
-N \irn G(u) \,dx.
\end{aligned} 
\end{equation}

Supposing further that $N=2$, or $N=3$ and \eqref{g3} holds with $2<p< \frac{5}{2}$, 
then any weak solution of \ef{eq} 
is actually a solution of the original problem \ef{eq2}. 
\end{theorem}

\begin{remark} \label{rem:1.2}
In the case $N=3$ and $\frac{5}{2} \le p < 3$,
we cannot expect that $\phi_\l \in H^2(\R^3)$ if $q(u) \ne 0$.
In fact if $\phi_\l \in H^2(\R^3)$,
we must have $\gl^{p-1} \in L^2(\R^3)$
because, roughly speaking, $g(s)$ behaves like $s^{p-1}$ at infinity.
However if $\frac{5}{2} \le p < 3$, 
it follows that $\gl^{p-1} \not\in L^2(\R^3)$.
In other words, when $N=3$ and $\frac{5}{2} \le p < 3$,
any weak solution of \ef{eq} cannot be a solution of the original problem \ef{eq2}
unless $q(u)=0$.
\end{remark}

We remark that the proof of the Pohozaev identity \eqref{eq:5.6}
is not straightforward because of the singularity of solutions of \ef{eq2}.
Indeed as is well-known, 
the Pohozaev identity can be obtained 
if we multiply the equation by $x \cdot \nabla u$. 
However since $u$ is singular, it is not clear 
whether all terms are integrable.
Especially we need to take care of the singularity of 
$\nabla \phi_\l$ and $\nabla \gl$ near the origin.
To overcome this difficulty,
a key is to establish the pointwise estimate of $| \nabla \phi_\l (x)|$ near the origin. 
As we will see in Lemma \ref{lem:5.4} below,
$| \nabla \phi_\l (x)|$ is unbounded at the origin when $N=3$.
Nevertheless, we are able to prove the convergence of
all terms to obtain the Pohozaev identity; 
see Lemma \ref{lem:5.5}.
Moreover we notice that the Pohozaev identity can be obtained 
by computing 
$\frac{d}{dt} I \left. \left( u \left( \frac{\cdot}{t} \right) \right) \right|_{t=1}=0$
formally; see Remark \ref{rem:2.6}.
We also mention that the Pohozaev identity for the case $N=2$ 
and for the power nonlinearity has been firstly obtained in \cite[Lemma 3.2]{FN}.
In this regard, \ef{eq:5.6} can be seen as a generalization of the result in \cite{FN}.

\medskip
The second main result of this paper concerns the existence 
of a nontrivial weak solution of \eqref{eq2} with positive charge.

\begin{theorem}\label{main}
Suppose that $N=3$ and $\alpha>0$.
Assume \eqref{g1}--\eqref{g4}. 
Then there exists a nontrivial weak solution $u_0 = \phi_\l+q(u_0)\cg_\l
\in \Ha$ of \ef{eq} with $q(u_0)>0$.
\end{theorem}

\begin{theorem} \label{thm:1.4}
Suppose that $N=3$, $\alpha \red{\le} 0$ or $N=2$.
Assume \eqref{g1}--\eqref{g3} and \eqref{g5}. 
Then there exists a nontrivial weak solution $u_0 = \phi_\l+q(u_0)\cg_\l
\in \Ha$ of \ef{eq} with $q(u_0)>0$.
\end{theorem} 

Here we briefly explain our main ideas of the proof.
We prove Theorem \ref{main} by applying the mountain pass theorem.
In fact under \ef{g1}-\ef{g4}, 
one can see that the functional $I$ has the mountain pass geometry. 
The existence of a non-trivial critical point of $I$ can be
shown by establishing the Palais-Smale condition.
Indeed once we could have the {\it boundedness of Palais-Smale sequences} in hand, 
one can expect the strong convergence of Palais-Smale sequences 
by introducing an auxiliary nonlinear term as in \cite{AP, BL, HIT, PW1, PW2}
and restricting ourselves to the space of radial functions.
However, as is well-known, the most difficult part is 
to prove the boundedness of Palais-Smale sequences.

In order to guarantee the existence of a bounded Palais-Smale sequence,
a standard strategy is to apply so-called {\it monotonicity trick} as in \cite{jj,S}.
However in the process of obtaining the boundedness, 
one needs to use the Pohozaev identity, which could require a lot of effort. 
Another approach, developed in \cite{CT, HIT, jj}, 
consists in considering a functional with an additional one dimensional variable. 
This guarantees the existence of a special Palais-Smale sequence 
which {\em almost} satisfies  the Pohozaev identity. 
In our case, even if we have already  obtained the Pohozaev identity,
this does not immediately lead us to obtain a bounded Palais-Smale sequence.
Indeed if we evaluate $I$ on the Pohozaev \emph{manifold}, 
using the identity \ef{eq:5.6}, 
we find that 
\[
I(u) =\frac 1N \| \nabla \phi_\l \|_2^2 + \frac \l N( \| \phi_\l \|_2^2 - \| u \|_2^2)
+ \frac{4-N}{2N} ( \alpha + \xi_\l) |q(u)|^2 
+ \frac \l N\| \gl \|_2^2 |q(u)|^2,
\]
for any $u= \phi_\l + q(u) \gl\in \Ha$. 
Therefore, if we take a sequence $\{u_n\}$ therein, 
with $u_n= \phi_{\l,n} + q(u_n) \gl$, 
because of the second term of the expression on $I$, 
the boundedness of $\| \nabla \phi_{\l,n }\|_2$ and of $q(u_n)$ cannot be derived 
from the above formula and hence
the application of the monotonicity trick does not work straightforwardly.
Moreover as we will see in Section 2,
the spatial scaling $x \to \frac{x}{t}$ makes a change in the parameter
$\l \to \frac{\l}{t^2}$.
This fact causes a difficulty of deriving the boundedness of 
Palais-Smale-Pohozaev sequences as in \cite{CT, HIT, jj}.
To overcome these difficulties,
we still use the technique of {\em adding one dimensional space} mentioned before 
but an additional blow-up type argument is necessary.
To carry out this procedure,
the restriction $N=3$ and $\alpha>0$ is needed under the assumptions \ef{g1}-\ef{g4}.
See Remark \ref{rem:5.7} for more detail about the necessity of this restriction.
Unfortunately, whenever $N=3$, $\alpha \red{\le} 0$ or $N=2$, 
the previous arguments do not work under the assumptions \ef{g1}-\ef{g4}. 
Therefore, in this case, in place of \eqref{g4}, we have to require \eqref{g5}. 
Observe that, under this growth condition, the situation is more straightforward. 
In particular,  the auxiliary functional $J$ is no more necessary 
and we can directly deal with classical Palais-Smale sequences.

Once we have proved the existence of a nontrivial solution of \ef{eq},
the most important ingredient is to show that
its singular part is not zero,
otherwise the obtained solution may coincide with that of \cite{BL}.
For that purpose, we take into account of 
the variational characterization 
and qualitative properties of ground state solutions of
the scalar field equation
\begin{equation} \label{scalar}
-\Delta u = g(u) \quad \text{in} \ \RN
\end{equation}
in the complex-valued setting.
We will see in Proposition \ref{prop:4.9} that
if the mountain pass solution $u= \phi_\l + q(u) \gl$ of \ef{eq}
satisfies $q(u)=0$, 
then $\phi_\l$ is a ground state solution of \ef{scalar},
contradicting to the boundary condition \ef{bdry}.

\medskip
This paper is organized as follows.
In Section \ref{se:fs}, we prepare several basic tools,
including some properties of the Green function 
and detailed informations of the decomposition of $u \in \Ha$.
In Section \ref{se:prop}, we prove some qualitative properties of nontrivial solutions
and establish the Pohozaev identity for \ef{eq},
then Theorem \ref{main2} will follows easily. 
Section \ref{se:vs} is devoted to the variational formulation of \ef{eq}.
Finally, we obtain the existence of a nontrivial solution of \ef{eq},
proving Theorem \ref{main} and Theorem \ref{thm:1.4},
by applying the mountain pass theorem in Section \ref{se:ex}. 
In the former case, as previously explained, the technique of
adding one dimensional space is necessary.

\section{Functional setting}\label{se:fs}

In this section, we prepare several basic tools,
including some properties of the Green function and
detailed informations of the decomposition of $u \in \Ha$.

First we recall some basic properties of the Green function $\gl$
of $-\Delta \gl + \l \gl = \delta_0$,
which is explicitly written as
\begin{equation}\label{eq:2.1}
\gl(x)= \mathcal{F}^{-1} \left( \frac{1}{|\xi|^2+\lambda} \right)
= \begin{cases}  
\dfrac{e^{-\sqrt{\lambda}|x|}}{4\pi |x|} &\  (N=3),\smallskip \\
\dfrac{K_0(\sqrt{\lambda}|x|)}{2\pi} &\  (N=2),
\end{cases}
\end{equation}
where $\mathcal{F}^{-1} $ is the inverse of the Fourier transform 
and $K_0$ is the modified Bessel function of the second kind of order $0$.

\begin{proposition} \label{prop:2.1}
Suppose $\lambda >0$ and $N=2, 3$. Then the following properties hold.

\begin{enumerate} [label=(\roman{*}),ref=\roman{*}]

\item \label{2.1-1}
$\gl \in L^{p}(\Omega) \cap L^{\infty}(\Omega)$
for any \red{$\Omega\subset \RN$ such that $0\notin\bar \Omega$} and $p \ge 1$. 

\item \label{2.1-2}
$\gl \in L^p(\mathbb{R}^N)$ for 
$\begin{cases}
1 \le p <3 &\  (N=3), \\
1 \le p <\infty &\  (N=2). 
\end{cases}$

\item \label{2.1-3}
$\gl \notin H^{1} (\mathbb{R}^N)$ and 
$x \cdot \nabla \gl \notin H^{1} (\mathbb{R}^N)$.

\item \label{2.1-4}
$\gl (x/t) =t^{N-2} \mathcal{G}_{\lambda/t^2}(x)$, 
for $t>0$ and $x \ne 0$. 

\item \label{2.1-5}
We have that
\[
\l \| \gl \|_2^2 = 
\begin{cases}
\dfrac{\xi_{\l}}{2} &(N=3),
\\[3mm]
\dfrac{1}{4 \pi} &(N=2).
\end{cases}
\]

\item For $\l_1$, $\l_2>0$, 
$\mathcal{G}_{\l_1}-\mathcal{G}_{\l_2}$ belongs to $H^2(\RN)$.
\end{enumerate}

\end{proposition}

Next we decompose $\gl$ as 
\[
\gl(x) = \gr(x) + \gs(x),
\]
where $\gs$ is the fundamental solution of $-\Delta$, that is,
\begin{equation} \label{eq:2.2}
\gs(x)= \mathcal{F}^{-1} \left( \frac{1}{|\xi|^2} \right)
= \begin{cases}
\dfrac{1}{4\pi |x|} &\ (N=3), \smallskip \\
-\dfrac{\log |x|}{2\pi} &\ (N=2).
\end{cases}
\end{equation}
From \ef{eq:2.1}-\ef{eq:2.2}, it is clear that $\gr \in C(\R^3)$ and
\[
\gr(0)= - \xi_{\l} =
\begin{cases}
-\dfrac{\sqrt{\l}}{4\pi} &\ (N=3), \smallskip \\
-\dfrac{\log\left(\frac{\sqrt{\l}}{2}\right)+\g}{2\pi} &\ (N=2).
\end{cases}
\]
We also note that $\gs$ is independent of $\lambda$. 
By the definition of $\gs$, we immediately have the following.

\begin{lemma} \label{lem:2.2} \

\begin{enumerate} [label=(\roman{*}),ref=\roman{*}]

\item \label{2.2-1}
When $N=3$, $\gs(x)$ satisfies
\[
\begin{aligned}
x \cdot \nabla \gs(x) &=-\frac{1}{4 \pi|x|}= -\gs(x) \quad (x \ne 0), \\
\nabla \big( x \cdot \nabla \gs(x) \big)
&= \frac{x}{4 \pi |x|^3} \quad (x \ne 0).
\end{aligned}
\]

\item \label{2.2-2}
When $N=2$, $\gl(x)$ satisfies
\[
\begin{aligned}
x \cdot \nabla \gl(x) &= O(1) \quad (|x| \sim 0), \\
\nabla \big(x \cdot \nabla \gl(x) \big) &= O(1) \quad ( |x| \sim 0).
\end{aligned}
\]
\end{enumerate}
\end{lemma}

Next we investigate the decomposition of $u \in H^1_{\alpha}(\R^N)$ in detail.

\begin{lemma} \label{lem:2.3}
Let $u \in \Ha$ be given. 
Then $q(u)$ does not depend on the choice of $\l >0$ 
and so it is determined uniquely.
\end{lemma}

\begin{proof}
Let $\l_1$, $\l_2 >0$ with $\l_1\neq\l_2$ be given and consider the decomposition:
\[
u=\phi_{\lambda_1}+q_{\l_1}(u) \mathcal{G}_{\l_1} \quad \text{and} \quad
u=\phi_{\lambda_2}+q_{\l_2}(u) \mathcal{G}_{\l_2}.
\]
Then one has
\[
\phi_{\l_1}-\phi_{\l_2}= q_{\l_2}(u)\mathcal{G}_{\l_2}-q_{\l_1}(u)\mathcal{G}_{\l_1}.
\]

Assume by contradiction that $q_{\l_1}(u)\neq q_{\l_2}(u)$. 
By the Plancherel theorem, it follows that
\[
\left|\n \big(q_{\l_2}(u)\mathcal{G}_{\l_2}-q_{\l_1}(u)\mathcal{G}_{\l_1}\big)\right|
\in L^2(\R^N) \Longleftrightarrow
|\xi|\left|\frac{q_{\l_1}(u)}{|\xi|^2+\l_1}-\frac{q_{\l_2}(u)}{|\xi|^2+\l_2}\right|\in L^2(\R^N),
\]
but the last one does not hold if $q_{\l_1}(u)\neq q_{\l_2}(u)$
by Proposition \ref{prop:2.1}-\eqref{2.1-3}. 
This implies that 
$q_{\l_2}(u)\mathcal{G}_{\l_2}-q_{\l_1}(u)\mathcal{G}_{\l_1}\notin \H$. 
Therefore $\phi_{\l_1}-\phi_{\l_2}$ does not belong to $\H$,
which is inconsistent, concluding the proof.
\end{proof}

\begin{remark} \label{rem.2.4}
If $u \in D(-\Delta_\a)$, 
we can give a precise expression of $q(u)$ as follows:
\begin{equation} \label{eq:2.3}
q(u) = \lim_{ |x| \to 0} \frac{u(x)}{\gs(x)}.
\end{equation}
\end{remark}

\begin{lemma} \label{lem:2.5}
Let $u$, $v \in \Ha$ be given and $t>0$, then the following holds:
\begin{enumerate} [label=(\roman{*}),ref=\roman{*}]

\item\label{2.5-1} $q(u+tv)= q(u) +t q(v)$,

\item\label{2.5-2} $q \big( u (\cdot/{t} ) \big)
= t^{N-2} q(u)$.

\end{enumerate}
\end{lemma}

\begin{proof}
\eqref{2.5-1} follows by the uniqueness result of  Lemma \ref{lem:2.3}.

For \eqref{2.5-2}, if $u=\phi_{\l}+q(u)\gl$,  
we have by Proposition \ref{prop:2.1}-\ef{2.1-4} that
\begin{align} \label{eq:2.4}
u(x/t) &=\phi_{\l}(x/t)+q(u)\gl(x/t)
= \phi_{\l}(x/t)+q(u) t^{N-2} \mathcal{G}_{\lambda/t^2}(x) \notag \\
&\stackrel{\l = t^2\mu}{=}
 \phi_{t^2 \mu} (x/t) + q(u)t^{N-2} \mathcal{G}_{\mu}(x)
\end{align}
and we conclude, once again, by the uniqueness of $q(u)$.
\end{proof}

\begin{remark} \label{rem:2.6}
For $u= \phi_\l + q(u) \gl \in D(-\Delta_{\alpha})$, let us denote
by $\eta_{t,\l}$ the regular part of $u(\frac{\cdot}{t})$, namely
\[
u\left(x/t\right) =\eta_{t, \lambda}(x)
+q\big( u ( \cdot/t ) \big) \gl(x).
\]
We emphasize that \ef{eq:2.4} shows that
\begin{equation*} 
\eta_{t, \l}(x)
= \phi_{t^2 \l} \left(x/t\right)
\neq \phi_\lambda\left(x/t\right).
\end{equation*}
In particular, under the transformation $x \to x/t$,
we have $\l \to \l/t^2$ and $q \to t^{N-2}q$.
From \ef{eq:2.4}, we also find that
\begin{align*}
I \big( u \left( \cdot/t \right) \big) 
&= \frac{1}{2}\big\langle-\Delta_\alpha u \left( \cdot/t \right), 
u \left( \cdot/t \right) \big\rangle 
-\int G \big(u \left( \cdot/t \right) \big) d x \notag \\
&= \frac{1}{2} \|\nabla \phi \left( \cdot/t \right) \|_{2}^2
+\frac{\lambda}{2t^2} \big( \|\phi \left( \cdot/t \right) \|_{2}^2
-\|u \left( \cdot/t \right) \|_{2}^2\big) \notag + \frac{1}{2} 
\left(\alpha+\xi_{\lambda / t^2}\right)\left| t^{N-2}q(u)\right|^2 \\
&\qquad\qquad -\int_{\mathbb{R}^N} G \big( u \left( \cdot/t \right) \big) d x \notag \\
&= \frac{t^{N-2}}{2} \|\nabla \phi\|_{2}^2
+ \frac{t^{N-2} \lambda}{2} \big(\|\phi\|_{2}^2-\|u\|_{2}^2\big) \notag 
+ \frac{t^{2(N-2)}}{2} \left(\alpha+\xi_{\lambda / t^2}\right) |q(u)|^2 
-t^N \!\int_{\mathbb{R}^N} \!G(u) d x.
\end{align*}
Moreover by the definition of $\xi_\l$ in \ef{xi} and 
Proposition \ref{prop:2.1}-\ef{2.1-5}, 
it follows that $$\frac{d}{dt} \xi_{\lambda / t^2} \big|_{t=1} = -2 \l \| \gl \|_2^2.$$
Thus by differentiating $I \big( u ( \cdot/t ) \big)$ at $t=1$,
we obtain the right hand side of \ef{eq:5.6}.
In other words, we are able to derive the Pohozaev identity \ef{eq:5.6}
from 
\[
\frac{d}{dt} I \big( u ( \cdot/t ) \big) \big|_{t=1}=0
\]
formally. 
\end{remark}

\section{Properties of nontrivial weak solutions}\label{se:prop}

In this section, we establish several properties of 
nontrivial weak solutions of \ef{eq}. 
In particular, we prove that any solution of \eqref{eq} 
satisfies a Pohozaev type identity, which is independently interesting.

First by \eqref{g2} and \eqref{g3}, we deduce that, for suitable $c_1,c_2>0$,
\begin{align}
|g(s)| &\le c_1 s+ c_2 s^{p-1}, \quad \text{for} \ s \ge 0,\label{gstima}
\\
|G(s)| &\le \frac{c_1}2s^2+ \frac{c_2}p s^{p}, \quad \text{for} \ s \ge 0.\label{Gstima}
\end{align}
Thus from \eqref{gstima}, \eqref{Gstima} and 
by definition of the extension to the complex plane of $g$ and $G$, 
we find that
\begin{align}
|g(u)| = |g(|u|)| &\le c_1 |u| + c_2|u|^{p-1}, 
\quad \text{for} \ u \in \C, \label{gCstima} \\
|G(u)| =| G(|u|)| &\le  \frac{c_1}{2} |u|^2 + \frac{c_2}{p} |u|^p, 
\quad \text{for} \ u \in \C. \label{GCstima}
\end{align}

Now we begin with the following regularity result.

\begin{proposition} \label{prop:5.1}

Let $u \in \Ha$ be a nontrivial weak solution of \ef{eq}
and decompose $u= \phi_\l +q(u) \gl$,  for $\l >0$. 
Then the following properties hold:

\begin{enumerate} [label=(\roman{*}),ref=\roman{*}]
\item\label{5.1-1} $\phi_\l \in C^{1, \kappa}(\RN \setminus \{ 0 \})$ 
for some $\kappa \in (0,1)$.

\item\label{5.1-2}
$\phi_\l \in H^2(\RD) \cap C^{1, \kappa}_{\rm loc}(\R^2)$ 
for some $\kappa \in (0,1)$ if $N=2$;

\item \label{5.1-3}
$\phi_\l \in C^{0,\kappa}_{\rm loc}(\R^3)$ for some $\kappa \in (0,1)$ if $N=3$;

\item \label{5.1-4}
$\phi_\l \in H^2(\R^3)$ if $N=3$ and $2< p< \frac{5}{2}$.
\end{enumerate}

\end{proposition}

\begin{proof}
We apply the elliptic regularity theory to the equation: 
\begin{equation} \label{eq:5.1}
-\Delta \phi_\l = \l q(u) \gl +g \big(\phi_\l + q(u) \gl\big) =: f_\l.
\end{equation}
By \eqref{gCstima}, we deduce that
\begin{equation} \label{eq:5.2}
|f_\l| \le C \left( |\phi_\l | + |q(u)| \gl 
+|\phi_\l|^{p-1} + |q(u)|^{p-1} \gl^{p-1} \right) \quad \text{a.e. in} \ \R^N. 
\end{equation}

First by Proposition \ref{prop:2.1}-\eqref{2.1-1}, 
it follows that $f_\l \in L^q(\Omega)$ for any 
$\Omega \subsetneq \R^N \setminus \{0\}$ and $q \ge 2$,
from which we have 
$\phi_\l \in W^{2,q}_ {\rm loc}(\Omega) \hookrightarrow 
C^{1,\kappa}_ {\rm loc}(\Omega)$.
Next when $N=2$, we know that $\gl \in L^q(\R^2)$ for all $q \ge 2$
by Proposition \ref{prop:2.1}-\eqref{2.1-2}.
This implies that $f_\l \in L^q(\R^2)$ for any $q \ge 2$
and especially $f_\l \in L^2(\R^2)$.
Then by the elliptic regularity theory
and the bootstrap argument, 
one finds that $\phi_\l \in H^2(\RD) \cap C^{1,\kappa}_{\rm loc}(\R^2)$.

In the case $N=3$, we only have $\gl \in L^q(\R^3)$ for $1 \le q <3$.
Since $2<p<3$, we can take $q_0 \in \left( \frac{3}{2},3 \right)$ 
so that $1<(p-1)q_0 <3$.
Then it holds that $\gl$, $\gl^{p-1} \in L^{q_0}(\R^3)$ 
and hence $f_\l \in L^{q_0}_ {\rm loc}(\R^3)$. 
By the elliptic theory and the bootstrap argument, we then have 
$\phi_\l \in W^{2,q_0}_ {\rm loc}(\R^3) \hookrightarrow 
C^{0,\kappa}_{\rm loc}(\RT)$
because $q_0> \frac{3}{2}$.
Finally if $N=3$ and $2< p< \frac{5}{2}$, one finds that $2(p-1)<3$
and hence $\gl^{p-1} \in L^2(\R^3)$.
This yields that $f_\l \in L^2(\R^3)$ and $\phi_\l \in H^2(\R^3)$.
\end{proof}

\begin{remark} \label{rem:5.2}
In the case $N=3$ and $\frac{5}{2} \le p< 3$,
we cannot expect that $\phi_\l \in H^2(\R^3)$ in general
because $\gl^{p-1} \not\in L^2(\R^3)$.
Nevertheless, the boundary condition $\phi_\l (0)= (\alpha + \xi_\l)q(u)$
always makes sense by the regularity result of Proposition \ref{prop:5.1}-\ef{5.1-3}.
\end{remark}

\begin{lemma} \label{lem:5.3}
Let $u \in \Ha$ be a nontrivial weak solution of \ef{eq}, fix $\l>0$
and decompose $u= \phi_\l +q(u) \gl$.
Then $\phi_\l \not\equiv 0$
and $q(u)$ can be assumed to be a non-negative real number.
\end{lemma}

\begin{proof}
Since $u$ is a weak solution of \eqref{eq}, 
we have that $\phi_\l(0)=(\a +\xi_\l)q(u)$ by Proposition \ref{prop:5.1}.

Being $u$ nontrivial, if $\phi_\l \equiv 0$, then $q(u)\neq 0$. 
So, since one has $0=(\alpha + \xi_\l)q(u)$, 
we deduce that $\l = \omega_{\alpha}$.
On the other hand, we have 
\[
-\l q(u) \gl = g \big( q(u) \gl \big) \quad \text{for all} \ x\in \R^N \setminus \{0\}, 
\]
from which we deduce by \eqref{g2} that
\[
-\l = \limsup_{|x| \to \infty} \frac{ g \big( q(u) \gl(x) \big)}{q(u) \gl(x)}
=-\omega.
\]
This is a contradiction to the fact $\omega_{\alpha} < \omega$
and hence $\phi_\l \not \equiv 0$.

Next let us put 
\[
e^{i \theta} u= \tilde{\phi}_\l + q(e^{i \theta} u)\gl 
= \tilde{\phi}_\l + e^{i \theta} q(u) \gl \quad \text{for} \ \theta \in \R.
\]
By the gauge invariance of $g$,
multiplying \ef{eq:5.1} by $e^{i \theta}$, one finds that
\[
-\Delta \tilde{\phi}_\l - \l e^{i\theta} q(u) \gl
= g \left( \tilde{\phi}_\l + e^{i \theta}q(u) \gl \right).
\]
Choosing $e^{i \theta} = \frac{\overline{q(u)}}{|q(u)|}$ if $q(u) \ne 0$,
we have $e^{i \theta} q(u)= |q(u)|>0$
and hence we may assume that $q(u)$ is a non-negative real number.
\end{proof}

Our next step is to establish the Pohozaev identity 
corresponding to \ef{eq}.
For this purpose, we first prove the following 
pointwise estimate for the gradient near the origin. 

\begin{lemma} \label{lem:5.4}
Let $u \in \Ha$ be a nontrivial weak solution of \ef{eq}
and decompose $u= \phi_\l + q(u) \gl$ for $\l >0$.
Then for $\eps \in (0,1)$, it holds that
\[
\sup_{ |x| = \eps} | \nabla \phi_\l (x)| =
\begin{cases}
O(\eps^{2-p}) &\text{if} \ \ N=3 \ \ \text{and} \ \ \frac{5}{2} \le p <3, \\
O(\eps^{-\frac{1}{2}}) &\text{if} \ \ N=3 \ \ \text{and} \ \ 2<p< \frac{5}{2}, \\
O(1) &\text{if} \ \ N=2.
\end{cases}
\]
\end{lemma}

\begin{proof}
By Proposition \ref{prop:5.1}-\eqref{5.1-2}, 
we know that $|\n \phi_{\l}|$ is locally bounded if $N=2$.
Thus it remains to consider the case $N=3$.

Now by Proposition \ref{prop:5.1}-\eqref{5.1-1}, one knows that
$\phi_\l \in C^1_ {\rm loc}(\Omega)$, 
for any $\Omega \subset \R^3 \setminus \{0\}$.
Thus from \ef{eq:5.1}, we can write $\phi_{\l}$ as
\[
\phi_\lambda(x)
=\frac{1}{4 \pi} \int_{\R^3} \frac{f_\lambda(y)}{|x-y|} \,d y 
\quad \text{for} \ x \in \mathbb{R}^3 \setminus \{0 \} 
\]
and
\[
\frac{\partial \phi_\lambda}{\partial x_i}(x)
=-\frac{1}{4 \pi} \int_{\mathbb{R}^3} \frac{x_i-y_i}{|x-y|^3} f_\lambda(y) \,d y, 
\quad (i=1,2,3) .
\]
Especially for $|x| = \eps$, one has
\[
\left| \frac{\partial \phi_\lambda}{\partial x_i} (x) \right| 
\le \frac{1}{4 \pi} \int_{\mathbb{R}^3} \frac{| f_\lambda(y)|}{|x-y|^2} \,dy
\]
and hence 
\begin{equation} \label{eq:5.3}
\sup _{|x| = \eps} | \nabla \phi_\l (x) |
\le \frac{\sqrt{3}}{4 \pi} \sup _{|x| = \eps} \int_{\mathbb{R}^3} 
\frac{|f_\l(y)|}{|x-y|^2} \,d y.
\end{equation}
Next we estimate the convolution term as follows.
\[
\begin{aligned}
\int_{\R^3} \frac{|f_\l(y)|}{|x-y|^2} \,d y
&=\int_{\{|x-y| \le \frac{\varepsilon}{2}\}} \frac{|f_\l(y)|}{|x-y|^2} \,d y
+\int_{\{\frac{\varepsilon}{2} \le |x-y|, \ |y| \le \varepsilon \}} 
\frac{|f_\l(y)|}{|x-y|^2} \,d y 
+\int_{\{\frac{\varepsilon}{2} \le |x-y|, \ 1 \le |y|\}} \frac{|f_\l(y)|}{|x-y|^2} \,d y \\
&\quad +\int_{\{\frac{\varepsilon}{2} \le |x-y| \le 1, \ \eps \le |y| \le 1\}}
\frac{|f_\lambda(y)|}{|x-y|^2} \,d y 
+\int_{\{ 1 \le |x-y|, \ \eps \le |y| \le 1\}} \frac{|f_\l(y)|}{|x-y|^2} \,d y \\
&=: ({\rm I})+ ({\rm II}) + ({\rm III}) + ({\rm IV}) + ({\rm V}).
\end{aligned}
\]
We note that by \ef{eq:2.2}, \ef{eq:5.2} and Proposition \ref{prop:5.1}-\eqref{5.1-1},
it follows that
\begin{equation} \label{eq:5.4}
\left|f_\lambda(y)\right| \le \frac{C}{|y|^{p-1}} 
\quad \text { for } \ y \in B_1 \setminus \{0 \} \ 
\text{and some} \ C>0.
\end{equation}
Moreover since $\gl$ decays exponentially at infinity,
$\phi_\l \in H^1(\R^3)$ and $2(p-1) < 6$,
we also have from \ef{eq:5.2} that
\begin{equation} \label{eq:5.5}
f_\l \in L^2\big(\{|y| \ge 1\}\big).
\end{equation}

First we observe that if $|x-y| \le \frac{\eps}{2}$ and $|x|= \eps$, then
\[
\frac{\eps}{2} \le |x| -|x-y| \le |y| \le |x| +|x-y| 
\le \frac{3\eps}{2}.
\]
Thus from \ef{eq:5.4}, one has
\[
({\rm I} )\le C \varepsilon^{1-p} \int_{\{|x-y| \le 
\frac{\eps}{2}\}} \frac{1}{|x-y|^2} \,d y =O\left(\varepsilon^{2-p}\right) .
\]
By using \ef{eq:5.4} and from $2<p<3$, we also have
\[
({\rm II}) \le C \varepsilon^{-2} \int_{\{|y| \le \varepsilon\}} 
\frac{1}{|y|^{p-1}} \,d y =O\left(\varepsilon^{2-p}\right) .
\]
Next by the Schwarz inequality and \ef{eq:5.5}, it holds that
\[
(\text {III}) \le
\left(\int_{\{1 \le |y|\}}\left|f_\lambda(y)\right|^2 \,d y\right)^{\frac{1}{2}}
\left(\int_{\{\frac{\eps}{2} \le |x-y|\}} \frac{1}{|x-y|^4} \,d y\right)^{\frac{1}{2}} 
=O\left(\varepsilon^{-\frac{1}{2}}\right). 
\]
From \ef{eq:5.4}, one also finds that
\[
\begin{aligned}
(\rm{IV}) & \le \left(\int_{\{\varepsilon \le |y| \le 1\}}
\left|f_\l(y)\right|^2 \,d y\right)^{\frac{1}{2}}
\left(\int_{\{\frac{\varepsilon}{2} \le |x-y| \le 1\}} 
\frac{1}{|x-y|^{4}} \,d y\right)^{\frac{1}{2}} \\
&\le C \left(\int_{\varepsilon}^1 r^{4-2p} \,d r\right)^{\frac{1}{2}}
\left(\int_{\frac{\eps}{2}}^1 r^{-2} \,d r\right)^{\frac{1}{2}}
=O\left(\varepsilon^{2-p}\right).
\end{aligned}
\]
Finally using \ef{eq:5.4}, we obtain
\[
({\rm V}) \le \left(\int_{\{\varepsilon \le |y| \le 1\}}
\left|f_\l(y)\right|^2 \,dy \right)^{\frac{1}{2}}
\left(\int_{\{1 \le |x-y|\}} \frac{1}{|x-y|^4} \,d y\right)^{\frac{1}{2}} 
=O\left(\varepsilon^{\frac{5-2p}{2}}\right).
\]
Thus from \ef{eq:5.3}, we deduce that
\[
\sup_{|x| = \varepsilon} \left|\nabla \phi_\lambda(x)\right|
=O\left(\varepsilon^{2-p}\right)
+O\left(\varepsilon^{-\frac{1}{2}}\right)
+O\left(\varepsilon^{\frac{5-p}{2}}\right) .
\]
Noticing that
\[
2-p \le - \frac{1}{2} < \frac{5-2p}{2} \le 0 \ \ \text{if} \ \ \frac{5}{2} \le p <3
\quad \text{and} \quad 
-\frac{1}{2} < 2-p < 0 < \frac{5-2p}{2} \ \ \text{if} \ \ 2<p<\frac{5}{2}, 
\]
we conclude.
\end{proof}

Now we are ready to show the Pohozaev identity for \ef{eq}.

\begin{lemma} \label{lem:5.5}
Let $u \in \Ha$ be a nontrivial weak solution of \ef{eq}
and decompose $u= \phi_\l + q(u) \gl$ for $\l >0$.
Then $u$ satisfies the Pohozaev identity \eqref{eq:5.6}.
\end{lemma}

\begin{proof}
In the following, for brevity, we set $q:=q(u)$. 
We recall that $\phi_\l$ satisfies
\[
\Delta \phi_\lambda +\lambda q \gl
+g(\phi_\lambda+q \gl)=0 \quad \text {in} \ \mathbb{R}^N.
\]
Multiplying this equation by $x \cdot \nabla (\overline{\phi_\l + q \gl})$,
one has
\[
\begin{aligned}
\re \left\{ (x \cdot \nabla \overline{\phi_\lambda}) \Delta \phi_\lambda \right\}
&=\re \left\{ 
\operatorname{div}\big( (x \cdot \nabla \overline{\phi_\lambda}) 
\nabla \phi_\lambda\big)
-\nabla\left(x \cdot \nabla \overline{\phi_\lambda} \right) 
\cdot \nabla \phi_\lambda \right\},  \\
\re \left\{ x \cdot \nabla (\overline{\phi_\lambda+q \gl}) g(\phi_\lambda+q \gl) \right\}
&= \re \left\{ \operatorname{div} \big( G(\phi_\lambda+q \gl ) x \big)
-N G(\phi_\lambda+q \gl) \right\}, \\
\lambda |q|^2 \gl (x \cdot \nabla \gl) 
&=\frac{\lambda |q|^2}{2} x \cdot \nabla | \gl|^2
=\frac{\l {|q|}^2}{2} \operatorname{div} \big( |\gl|^2 x\big)
-\frac{N \lambda |q|^2}{2} | \gl|^2, \\
\re \left\{ \lambda q \gl (x \cdot \nabla \overline{\phi_\lambda}) \right\}
&= \re \left\{ \lambda q \operatorname{div}\left( \overline{\phi_\lambda} \gl x \right)
-N \lambda q \overline{\phi_\lambda} \gl 
-\lambda q \overline{\phi_\lambda} (x \cdot \nabla \gl) \right\}, \\
\re \left\{ \bar{q} (x \cdot \nabla \gl) \Delta \phi_\lambda \right\}
&= \re \left\{ 
\bar{q} \operatorname{div}\big((x\cdot \nabla \gl) \nabla \phi_\lambda\big)
- \bar{q} \nabla (x \cdot \nabla \gl ) \cdot \nabla \phi_\lambda \right\}.
\end{aligned}
\] 
Integrating them over $\{ x \in \R^N : \eps \le |x| \le R\}$
for $0<\eps <1 < R<+ \infty$,
using the divergence theorem
and taking the real part, we get
\begin{align} \label{eq:5.7}
&\re \int_{\{\eps \le |x| \le R\}} (x \cdot \nabla \overline{\phi_\lambda}) 
\Delta \phi_\lambda \,dx \\
&= \frac{N-2}{2} \int_{\{\eps \le |x| \le R\}} |\nabla \phi_\lambda |^2 \,d x 
+ \re \int_{\{|x|=R\}} 
(x \cdot \nabla \overline{\phi_\lambda}) (\nabla \phi_\lambda \cdot \nu) \,dS
+ \re \int_{\{|x| = \eps\} }
(x \cdot \nabla \overline{\phi_\lambda}) 
(\nabla \phi_\lambda \cdot \nu) \,dS\notag
\\
&\quad -\frac R2\int_{\{|x| = R\} }|\n \phi_{\l}|^2\, dS
+\frac \eps2\int_{\{|x| = \eps\} }|\n \phi_{\l}|^2\, dS, \notag
\end{align}
\begin{align} \label{eq:5.8}
& \re \int_{\{\eps \le |x| \le R\}} 
x \cdot \nabla (\overline{\phi_\lambda +q \gl}) g(\phi_\lambda+q \gl) \,dx \\
&=-N \int_{\{\eps \le |x| \le R\}} G(\phi_\lambda+q \gl) \,d x 
+ \int_{\{|x|=R\}} G (\phi_\lambda+ q \gl) (x \cdot \nu) \,dS
+ \int_{\{|x|=\varepsilon\}} G (\phi_\lambda+q \gl ) (x \cdot \nu) \,dS, \notag
\end{align}
\begin{align} \label{eq:5.9}
& \re \int_{\{\eps \le |x| \le R\}} \lambda |q|^2 \gl (x \cdot \nabla \gl) \,dx \\
&= -\frac{N \l |q|^2}{2} \int_{\{\varepsilon \le |x| \le R\}} |\gl|^2 \,d x 
+\frac{\l |q|^2}{2} \int_{\{|x| =R\}} |\gl|^2 (x \cdot \nu) \,dS
+\frac{\l |q|^2}{2} \int_{\{|x| = \eps\}} |\gl|^2 (x \cdot \nu) \,dS, \notag
\end{align}
\begin{align} \label{eq:5.10}
\re &\int_{\{\eps \le |x| \le R\}} \left\{ \l q \gl (x \cdot \nabla \overline{\phi_\l})
+\bar{q} (x \cdot \n \gl) \Delta \phi_\l \right\} \,dx \\
&= \re \Big\{ -N \lambda q \int_{\{\varepsilon \le |x| \le R\}} \overline{\phi_\l} \gl \,d x
-\lambda q \int_{\{\varepsilon \le |x| \le R\}} \overline{\phi_\lambda} 
(x \cdot \nabla \gl) \,d x 
-\bar{q} \int_{\{\varepsilon \le |x| \le R\}} 
\nabla (x \cdot \nabla \gl) \cdot \nabla \phi_\l \,d x \notag \\
&\quad 
+\lambda q \int_{\{|x|=R\}} \overline{\phi_\l} \gl (x \cdot \nu) \,dS
+\bar{q} \int_{\{|x|=R\}} (x \cdot \nabla \gl) 
(\nabla \phi_\lambda \cdot \nu) \,dS \notag \\
&\quad 
+ \lambda q \int_{\{|x| =\varepsilon\}} \overline{\phi_\lambda} \gl (x \cdot \nu) \,dS
+\bar{q} \int_{\{|x|=\varepsilon\}} (x \cdot \nabla \gl) 
(\nabla \phi_\lambda \cdot \nu) \,dS 
\Big\} \notag \\
&= \re \Big\{ -N \lambda q \int_{\{\varepsilon \le |x| \le R\}} \overline{\phi_\l} \gl \,d x
-\lambda q \int_{\{\varepsilon \le |x| \le R\}} \overline{\phi_\lambda} 
(x \cdot \nabla \gl) \,d x 
+\bar{q} \int_{\{\varepsilon \le |x| \le R\}} 
\Delta (x \cdot \nabla \gl ) \phi_\lambda \,d x \notag \\
&\quad 
+\lambda q \int_{\{|x|=R\}} \overline{\phi_\l} \gl (x \cdot \nu) \,dS
+\bar{q} \int_{\{|x|=R\}} (x \cdot \nabla \gl) (\nabla \phi_\lambda \cdot \nu) \,dS 
-\bar{q} \int_{\{|x|=R\}} \phi_\lambda \nabla (x \cdot \nabla \gl) \cdot \nu \,dS 
\notag \\
&\quad 
+ \lambda q \int_{\{|x| =\varepsilon\}} \overline{\phi_\lambda} \gl (x \cdot \nu) \,dS
+\bar{q} \int_{\{|x|=\varepsilon\}} (x \cdot \nabla \gl) 
(\nabla \phi_\lambda \cdot \nu) \,dS 
-\bar{q} \int_{\{|x|=\varepsilon\}} \phi_\lambda \nabla 
(x \cdot \nabla \gl) \cdot \nu \,dS
\Big\}. \notag
\end{align}
Here, for \eqref{eq:5.7}, we used that
\[
\begin{aligned}
&- \re \int_{\{\eps \le |x| \le R\}} \nabla (x \cdot \nabla \overline{\phi_\l}) 
\cdot \nabla \phi_\l \,dx
\\
& =-\int_{\{\varepsilon \le |x| \le R\}} 
\left\{ \frac{1}{2} \nabla (|\nabla \phi_\lambda |^2) \cdot x 
+| \nabla \phi_\l |^2 \right\} \,d x \\
&= -\int_{\{\varepsilon \le |x| \le R\}}
|\nabla \phi_\lambda|^2 \,d x
-\frac{1}{4} \int_{\{\eps \le |x| \le R\}} 
\nabla (|\nabla \phi_\lambda|^2) \cdot \nabla |x|^2 \,d x \\
&=-\int_{\{\varepsilon \le |x| \le R\}} |\nabla \phi_\l |^2 \,d x
+\frac{1}{4} \int_{\{\eps \le |x| \le R\}} |\nabla \phi_\l|^2 \Delta ( |x|^2 ) \,d x \\
&\quad 
-\frac{1}{4} \int_{\{|x|=R\}} |\nabla \phi_\l |^2 \nabla ( |x|^2) \cdot \nu \,d S 
-\frac{1}{4} \int_{\{|x|=\eps\}} |\nabla \phi_\l |^2 \nabla ( |x|^2) \cdot \nu \,d S \\
&= \frac{N-2}{2} \int_{\{\eps \le |x| \le R\}} |\nabla \phi_\lambda |^2 \,d x
-\frac R2\int_{\{|x| = R\} }|\n \phi_{\l}|^2\, dS
+\frac \eps2\int_{\{|x| = \eps\} }|\n \phi_{\l}|^2\, dS.
\end{aligned}
\]
Moreover since $\Delta \gl = \l \gl$ for $x \ne 0$, we also find that
\[
\begin{aligned}
\Delta\left(x \cdot \nabla \gl \right) 
& =\sum_{j=1}^N \frac{\partial^2}{\partial x_j^2}
\left(\sum_{i=1}^N x_i \frac{\partial \gl}{\partial x_i}\right) 
=\sum_{j=1}^N \frac{\partial}{\partial x_j}
\left(\sum_{i=1}^N x_i \frac{\partial^2 \gl}{\partial x_i \de x_j}
+\frac{\partial \gl}{\partial x_j}\right) \\
& =\sum_{i=1}^N \sum_{j=1}^N x_i 
\frac{\partial^3 \gl}{\partial x_i \partial x_j^2}
+2 \sum_{j=1}^N \frac{\partial^2 \gl}{\partial x_j^2} 
 =x \cdot \nabla (\Delta \gl)+ 2 \Delta \gl \\
&= \l x \cdot \nabla \gl + 2\l \gl.
\end{aligned}
\]
Thus one gets
\begin{equation} \label{eq:5.11}
\re \Big\{
-\lambda q \int_{\{\varepsilon \le |x| \le R\}} \overline{\phi_\lambda} 
(x \cdot \nabla \gl) \,d x 
+\bar{q} \int_{\{\varepsilon \le |x| \le R\}} 
\Delta (x \cdot \nabla \gl ) \phi_\lambda \,d x \Big\}
=2 \re \Big\{ \l \bar{q} \int_{\{\varepsilon \le |x| \le R\}} \phi_\l \gl \,d x \big\}.
\end{equation}
From \ef{eq:5.7}-\ef{eq:5.11}, we arrive at
\begin{align} \label{eq:5.12}
0&= \frac{N-2}{2} \int_{\{\eps \le |x| \le R\}} |\nabla \phi_\lambda |^2 \,d x
-N \int_{\{\eps \le |x| \le R\}} G(\phi_\lambda+q \gl) \,d x \\
&\quad -\frac{N \l |q|^2}{2} \int_{\{\varepsilon \le |x| \le R\}} |\gl|^2 \,d x 
-(N-2) \lambda \re 
\Big\{ \bar{q} \int_{\{\varepsilon \le |x| \le R\}} \phi_\l \gl \,d x \Big\}
+C_1(R) + C_2(\eps), \notag
\end{align}
where
\begin{align*}
&C_1(R) \\
&:= \re \Big\{
\int_{\{|x|=R\}} 
(x \cdot \nabla \overline{\phi_\lambda}) (\nabla \phi_\lambda \cdot \nu) \,dS
+ \int_{\{|x|=R\}} G (\phi_\lambda+ q \gl) (x \cdot \nu) \,dS
+\frac{\l |q|^2}{2} \int_{\{|x| =R\}} |\gl|^2 (x \cdot \nu) \,dS \\
&\quad +\lambda q \int_{\{|x|=R\}} \overline{\phi_\l} \gl (x \cdot \nu) \,dS
+\bar{q} \int_{\{|x|=R\}} (x \cdot \nabla \gl) (\nabla \phi_\lambda \cdot \nu) \,dS 
- \bar{q} \int_{\{|x|=R\}} \phi_\lambda
\nabla (x \cdot \nabla \gl) \cdot \nu \,dS 
\\
&\quad -\frac R2\int_{\{|x| = R\} }|\n \phi_{\l}|^2\, dS\Big\} ,
\end{align*}
\begin{align*}
&C_2(\eps) \\
&:= \re \Big\{
\int_{\{|x| = \eps \}}
\left(x \cdot \nabla \overline{\phi_\lambda} \right) 
(\nabla \phi_\lambda \cdot \nu) \,dS
+ \int_{\{|x|=\varepsilon\}} G (\phi_\lambda+q \gl ) (x \cdot \nu) \,dS
+ \frac{\l |q|^2}{2} \int_{\{|x| = \eps\}} |\gl|^2 (x \cdot \nu) \,dS \\
&\quad 
+ \lambda q \int_{\{|x| =\varepsilon\}} \overline{\phi_\lambda} \gl (x \cdot \nu) \,dS
+\bar{q} \int_{\{|x|=\varepsilon\}} (x \cdot \nabla \gl) 
(\nabla \phi_\lambda \cdot \nu) \,dS 
- \bar{q} \int_{\{|x|=\varepsilon\}} \phi_\lambda
\nabla (x \cdot \nabla \gl) \cdot \nu \,dS
\\
&\quad +\frac \eps2\int_{\{|x| = \eps\} }|\n \phi_{\l}|^2\, dS\Big\} .
\end{align*}
Since $\phi_\l \in H^1(\R^N)$, 
$G(\phi_\l + q \gl) \in L^1(\R^N)$,
$\gl \in L^2(\R^N)$
and $\nabla \gl$ decays exponentially at infinity,
arguing as in \cite[P. 321, Proof of Proposition 1]{BL}, 
it follows that
\begin{equation} \label{eq:5.13}
C_1(R) \to 0 \quad \hbox{as} \ \ R \to +\infty.
\end{equation}

Next we recall that $\phi_{\l} \in L^{\infty}_{\rm loc}(\R^N)$
and on the set $\{x\in \RN : |x|=\eps\}$, we have
\[
\gl = 
\begin{cases} 
O(\eps^{-1}) &(N=3), \\
O(|\log \eps|) &(N=2).
\end{cases}
\]
Then one finds that
\begin{align*}
\int_{\{|x| = \eps\}} |\gl|^2 (x \cdot \nu) \,dS
&= \begin{cases} 
O(\eps) &(N=3), \\
O(\eps^2 |\log \eps|^2) &(N=2),
\end{cases} \\
\int_{\{|x| = \eps\}} |\gl|^p (x \cdot \nu) \,dS
&= \begin{cases} 
O(\eps^{3-p}) &(N=3), \\
O(\eps^2 |\log \eps|^p) &(N=2),
\end{cases} \\
\int_{\{|x| =\varepsilon\}} \overline{\phi_\lambda} \gl (x \cdot \nu) \,dS
&= \begin{cases} 
O(\eps^2) &(N=3), \\
O(\eps^2 |\log \eps|) &(N=2).
\end{cases}
\end{align*}
Moreover by \eqref{GCstima}, we find that
\[
\int_{\{ |x|=\varepsilon \}}\!\! G (\phi_\lambda+q \gl ) (x \cdot \nu) \,dS
+ \frac{\l |q|^2}{2} \int_{ \{ |x| = \eps \}}\!\! |\gl|^2 (x \cdot \nu) \,dS 
+ \l q \int_{\{ |x| =\varepsilon \}} \!\!\overline{\phi_\lambda} \gl (x \cdot \nu) \,dS 
\to 0 \quad \text{as} \ \eps \to 0^+.
\]
Next by Lemma \ref{lem:5.4}
and using the fact that, on the set $\{x\in \RN : |x|=\eps\}$,
\[
x \cdot \nabla \gl = 
\begin{cases} 
O(\eps^{-1}) &(N=3), \\
O(1) &(N=2),
\end{cases}
\]
we have
\begin{align*}
\int_{\{|x| = \eps \}}
\left(x \cdot \nabla \overline{\phi_\lambda}\right) 
(\nabla \phi_\lambda \cdot \nu) \,dS
&= \begin{cases} 
O(\eps^{7-2p}) &\quad \text{if} \ \ N=3 \ \text{and} \ \frac{5}{2} \le p <3, \\
O(\eps^2) &\quad \text{if} \ \ N=3 \ \text{and} \ 2< p< \frac{5}{2}, \\
O(\eps^2) &\quad \text{if} \ \ N=2,
\end{cases} \\
\int_{\{|x| = \eps\} }
\left(x \cdot \nabla \gl \right) (\nabla \phi_\lambda \cdot \nu) \,dS
&= \begin{cases} 
O(\eps^{3-p}) &\quad \text{if} \ \ N=3 \ \text{and} \ \frac{5}{2} \le p <3, \\
O(\eps^{\frac{1}{2}}) &\quad \text{if} \ \ N=3 \ \text{and} \ 2< p< \frac{5}{2}, \\
O(\eps) &\quad \text{if} \ \ N=2,
\end{cases}
\\
\frac \eps 2\int_{\{|x| = \eps \}}
|\nabla \phi_\lambda|^2 \,dS
&= \begin{cases} 
O(\eps^{7-2p}) &\quad \text{if} \ \ N=3 \ \text{and} \ \frac{5}{2} \le p <3, \\
O(\eps^2) &\quad \text{if} \ \ N=3 \ \text{and} \ 2< p< \frac{5}{2}, \\
O(\eps^2) &\quad \text{if} \ \ N=2.
\end{cases} 
\end{align*}
Finally we show that
\begin{equation} \label{eq:5.14}
\int_{\{|x|=\varepsilon\}} \phi_\lambda
\nabla (x \cdot \nabla \gl) \cdot \nu \,dS
\rightarrow -(N-2) \phi_\lambda(0)
\quad \text { as } \ \varepsilon \rightarrow 0^+.
\end{equation}
Indeed, it suffices to consider $\gs$.
In the case $N=3$, it follows from 
Lemma \ref{lem:2.2}-\ef{2.2-1} and $\nu= - \frac{x}{|x|}$ on $|x|=\eps$ that 
\[
\left| \int_{\{|x|= \eps\}} \nabla ( x \cdot \nabla \gs) \cdot \nu \,dS \right| \le 1
\]
and 
\[
\int_{\{|x|=\eps\}} \phi_\l (0) \nabla (x \cdot \gs) \cdot \nu \,dS
= -\int_{\{|x| = \eps\}} \frac{ \phi_\l(0) }{4\pi |x|^2} \,dS 
= - \phi_\l(0).
\]
Thus we have
\[
\begin{aligned}
\left| \int_{\{|x|=\varepsilon\}} \phi_\lambda
\nabla (x \cdot \nabla \gs) \cdot \nu \,dS
+ \phi_\l(0)\,\right| 
&= \left| \int_{\{|x|=\eps \}} 
\big( \phi_\lambda(x) - \phi_\l(0) \big)
\nabla (x \cdot \nabla \gs ) \cdot \nu \,dS \right| \\
& \le \sup_{|x|=\varepsilon} 
\left| \phi_\l(x) - \phi_\l(0) \right| 
\rightarrow 0 \quad \text { as } \ \varepsilon \rightarrow 0^+. 
\end{aligned}
\]
When $N=2$, using Lemma \ref{lem:2.2}-\ef{2.2-2}, we also have
\[
\left| \int_{\{|x|=\varepsilon\}} 
\phi_\lambda \nabla(x \cdot \nabla \gl) \cdot \nu \,dS \right|
\le C \eps  \rightarrow 0 
\quad \text {as} \ \varepsilon \rightarrow 0^+,
\]
from which we deduce that \ef{eq:5.14} holds.

Now from \ef{eq:5.13} and \ef{eq:5.14}, 
passing a limit $R \to +\infty$ and $\eps \to 0^+ $ in \ef{eq:5.12},
we find that
\begin{align*} 
0&= \frac{N-2}{2} \| \nabla \phi_\l \|_2^2 
-N \int_{\R^N} G(\phi_\lambda+q \gl) \,d x \\
&\quad -\frac{N \l |q|^2}{2} \| \gl \|_2^2  
-(N-2) \lambda \re \langle \phi_\l , q \gl \rangle 
+(N-2) \re\left\{ \bar{q} \phi_\l(0) \right\}.
\end{align*}
Using 
\[
\| \phi_\l\|_2^2 - \| u \|_2^2
= -2 \re \langle \phi_\l, q \gl \rangle  - |q|^2 \| \gl \|_2^2
\quad \text{and} \quad \phi_\l(0)= (\alpha + \xi_\l)q,
\]
we obtain \ef{eq:5.6}.
\end{proof}

\begin{remark} \label{rem:5.6}
Observe that in the case $N=3$, by Proposition \ref{prop:2.1}-\eqref{2.1-5}, 
the Pohozaev identity \eqref{eq:5.6} can be written also in the following way
\begin{equation*}
\frac{1}{2}\| \nabla \phi \|_2^2
+\frac{\lambda}{2}\|\phi \|_2^2- \frac{\lambda}{2}\| u \|_2^2
+\frac{1}{2}(\alpha+\xi_\l ) | q(u) |^2 
+\frac{1}{2} \alpha |q(u)|^2  -3 \int_{\R^3} G(u) \,dx =0.
\end{equation*}
\end{remark}

Now Theorem \ref{main2} 
follows by Proposition \ref{prop:5.1}, Lemma \ref{lem:5.3}
and Lemma \ref{lem:5.5}.

\section{Variational setting}\label{se:vs}

In this section, we introduce a variational setting of \ef{eq}.
Let us define the functional $I:\Ha\to \R$ as
\begin{equation*}
I(u)= \frac 12\|\n \phi_\l\|_2^2+\frac \l 2\|\phi_\l \|_2^2-\frac \l 2\|u\|_2^2
+\frac 12(\a +\xi_\l)|q(u)|^2
- \irn G(u) \,dx,
\end{equation*}
for $\lambda >0$ and 
$u=\phi_\l+q(u)\cg_\l\in \Ha$.
Clearly, if $u\in \H$, then 
\begin{equation*}
I(u)= \frac 12\|\n u\|_2^2
 - \irn G(u) \,dx.
\end{equation*}
By Proposition \ref{prop:2.1}-\eqref{2.1-2} and \eqref{GCstima},
one can see that $I$ is well defined on $\Ha$.
Next we show that $I$ is actually of the class $C^1$.
Although this seems to be standard,
we need to be careful in the case $N=3$ 
because of the less integrability of $\gl$. 

\begin{proposition} \label{prop:2.6}
The functional $I$ is of the class $C^1$ on $\Ha$. 
Moreover for any $u,v\in \Ha$, if $u=\phi_{\l}+q(u)\gl$ and $v=\psi_{\l}+q(v)\gl$, 
it follows that
\begin{align}\label{derivata}
I'(u)[v ]
&= \re \left\{ \langle \nabla \phi_\l, \nabla \psi_\l \rangle 
+ \l \langle \phi_\l, \psi_\l \rangle - \l \langle u,v \rangle 
+(\alpha + \xi_\l ) q(u) \overline{q(v)} 
- \irn g(u) \bar{v} \,dx \right\}. 
\end{align}
\end{proposition}

\begin{proof}
First we put $I(u)=I_1(u) + I_2(u)$ with
\[
\begin{aligned}
I_1(u) &:=\frac{1}{2} \|\nabla \phi_\lambda \|_2^2
+\frac{\lambda}{2} \|\phi_\lambda \|_2^2 -\frac{\lambda}{2}\| u\|_2^2
+\frac{1}{2} \left(\alpha+\xi_\lambda\right) | q(u)|^2, \\
I_2(u) &:= 
\int_{\mathbb{R}^N} G(u) \,d x.
\end{aligned}
\]
For $v= \psi_\lambda+q(v) \gl \in \Ha$, we also define
$L_1$, $L_2 \in \mathcal{L} \left( \Ha, \mathbb{R}\right)$ by
\[
\begin{aligned}
L_1(u)[ v] &:= \re \left\{ 
\left\langle\nabla \phi_\l, \nabla \psi_\l \right\rangle
+\lambda \left\langle \phi_\l, \psi_\lambda\right\rangle
-\lambda \langle u, v \rangle 
+\left(\alpha+\xi_\l \right) q(u) \overline{q(v)} \right\},  \\
L_2(u)[ v] &:= 
\re \int_{\mathbb{R}^N} g(u) \bar{v} \,d x.
\end{aligned}
\]
From \eqref{gCstima},
one finds that $L_1$ and $L_2$ are both bounded on $\Ha$.
Moreover for $\l > \omega_{\alpha}$, 
we have by  Lemma \ref{lem:2.5}-\eqref{2.5-1} that
\[
\begin{aligned}
I_1(u+v) &= \frac{1}{2} \| \nabla \phi_\l + \nabla \psi_\l \|_2^2
+ \frac{\l}{2} \| \phi_\l + \psi_\l \|_2^2
-\frac{\l}{2} \| u+v \|_2^2 
+\frac{1}{2} (\alpha + \xi_\l ) | q(u)+q(v) |^2 \\
&=I_1(u)+L_1(u)[v ]
+\frac{1}{2}\| \nabla \psi_\l \|_2^2 +\frac{\lambda}{2} \| \psi_\l \|_2^2
-\frac{\lambda}{2} \| v \|_2^2 
+\frac{1}{2}(\alpha + \xi_\l ) | q(v) |^2 \\
&= I_1(u) +L_1(u)[v] + \frac{1}{2} \| v \|_{\Hal}^2 - \frac{\l}{2} \|v\|_2^2.
\end{aligned}
\]
Thus one has 
\[
\big|I_1(u+v)-I_1(u)-L_1(u)[v ]\big| =
\frac{1}{2} \| v \|_{\Hal}^2 - \frac{\l}{2} \|v\|_2^2
\le \frac 12  \| v \|_{\Hal}^2
\]
yielding that $I_1$ is Frechet differentiable and $I_1'=L_1$.
Note that the differentiability of $I_1$ is independent of the choice of $\l$
and the decomposition $u= \phi_\l + q(u) \gl$ because of \ef{equiv}.
Furthermore one has
\[
\left\| I_1^{\prime}(u) \right\| = 
\sup_{v \in \Ha, \,\| v \|_{\Hal} \le 1} \big| I_1'(u)[v] \big|
\le \| u \|_{\Hal}
\]
from which we can conclude that $I_1'$ is continuous on 
$\mathcal{L} \left( \Ha, \mathbb{R}\right)$. 

Next we prove that $I_2$ is of the class $C^1$.
For this purpose, we first claim that
\begin{equation} \label{eq:2.12}
\left|\frac{1}{t} \big\{I_2(u+t v)-I_2(u)-t L_2(u) [v] \big\} \right| 
\rightarrow 0, \quad \text { as } \ t \rightarrow 0^+,
\end{equation}
which implies that $I_2'=L_2$.
Indeed from \eqref{gCstima}, one has
\[
\begin{aligned}
&\left| \frac{1}{t} \big\{ G(u+tv)-G(v)-t g(u) \bar{v} \big\} \right| \\
&\le \sup_{t \in[0,1]} \big\{ |g(u+t v)|+|g(u)| \big\} |v|  \\
&\le C\left(|u|+|v|+|u|^{p-1}+|v|^{p-1}\right)|v| \\
&\le C \Big\{ | \phi_\l | + |\psi_\l | +(|q(u)|+|q(v)|) \gl \\
&\qquad + | \phi_\l |^{p-1} + | \psi_\l |^{p-1} + ( |q(u)|^{p-1} + |q(v)|^{p-1}) \gl^{p-1}
\Big\} ( |\psi_\l| + |q(v)|\gl) \\
&=: h_\l \quad \text{a.e.} \ x \in \R^N.
\end{aligned}
\]
Since $2<p<3$, if $N=3$, and $p>2$, if $N=2$,
it follows by Proposition \ref{prop:2.1}-\eqref{2.1-2} that $h_\l \in L^1(\R^N)$.
Moreover we have
\[
\frac{1}{t} \big\{ G(u+tv)-G(u)-t g(u) \bar{v} \big\} \rightarrow 0
\quad \text { a.e. } \ x \in \mathbb{R}^N \text { as } \ t \rightarrow 0.
\]
Thus by the Lebesgue dominated convergence theorem,
\ef{eq:2.12} follows.

Finally we prove that if $u_n \to u_0$ in $\Ha$, then
\begin{equation} \label{eq:2.13}
\sup_{v \in \Ha, \,\| v \|_{\Hal} \le 1}
\left| \big\{ I_2'(u_n) - I_2'(u_0) \big\} [v] \right| \to 0 \quad \text{as} \ n \to +\infty, 
\end{equation}  
from which we deduce that $I_2'$ is continuous.
Putting 
\[
u_n=\phi_{n,\l}+q(u_n) \gl \quad \text { and } \quad
u_0=\phi_{0,\lambda}+q(u_0) \gl,
\]
it holds that 
\[
\phi_{n,\lambda} \rightarrow \phi_{0, \lambda} \quad \text{in} \ H^1(\R^N)
\quad \text{and} \quad q(u_n) \rightarrow q(u_0).
\]
Especially one gets $\phi_{n,\l} \to \phi_{0,\l}$ in $L^2(\RN) \cap L^p(\R^N)$ 
and, up to a subsequence,
\begin{equation} \label{eq:2.14}
\left| \phi_{n,\l} \right| \le \Phi \ \text { a.e. in } \ \mathbb{R}^N, \ 
|q(u_n)| \le M \ \text { for all} \ n \in \N
\end{equation}
for some $\Phi \in L^2(\RN) \cap L^p(\R^N)$ and $M>0$.
Hereafter we write $\phi_{n,\l}=\phi_n$ and $\phi_{0,\l}=\phi_0$ for simplicity.
For any $R>0$, 
by using \eqref{gCstima} and \ef{eq:2.14}, we find that
\[
\begin{aligned}
&\left| \int_{\{|x| \ge R\}} \big\{ g(u_n)-g(u_0) \big\} \bar{v} \,d x \right| \\
&\le C \int_{\{|x| \ge R\}}
\left( |u_n|+|u_n|^{p-1}+|u_0|+|u_0|^{p-1} \right) |v| \,d x \\
&\le C \int_{\{|x| \ge R\}}
\left( |\phi_n |+ |q(u_n)| \gl +|u_0 |
+|\phi_n|^{p-1}+| q(u_n) |^{p-1} \gl^{p-1} + |u_0 |^{p-1} \right)
|v| \,dx \\
&\le C \left( \|\phi_n \|_{L^2(\{|x| \ge R\})} +|q(u_n)| \| \gl \|_{L^2(\{|x| \ge R\})}
+ \| u_0 \|_{L^2(\{|x| \ge R\})} \right) \| v \|_{L^2(\{|x| \ge R\})} \\
&\quad +C \left( \| \phi_n \|_{L^p(\{|x| \ge R\})}^{p-1}
+|q(u_n)|^{p-1} \| \gl \|_{L^p(\{|x| \ge R\})}^{p-1} 
+ \| u_0 \|_{L^p(\{|x| \ge R\})}^{p-1} \right) \| v \|_{L^p(\{|x| \ge R\})} \\
&\le C \left( \| \Phi \|_{L^2(\{|x| \ge R\})}
+ M \| \gl \|_{L^2(\{|x| \ge R\})} + \| u_0 \|_{L^2(\{|x| \ge R\})} \right) 
\| v \|_{\Hal} \\
&\quad + C \left(
\| \Phi \|_{L^p(\{|x| \ge R\})}^{p-1} + M^{p-1} \| \gl \|_{L^p(\{|x| \ge R\})}^{p-1}
+\| u_0 \|_{L^p(\{|x| \ge R\})}^{p-1} \right) \| v \|_{\Hal}.
\end{aligned}
\]
Thus for any $\eps >0$, there exists $R_{\eps}>0$ such that
\begin{equation} \label{eq:2.15}
\sup_{v \in \Ha, \, \| v \|_{\Hal} \le 1}
\left| \int_{\{|x| \ge R_{\eps}\}} \big\{ g(u_n) -g(u_0) \big\} \bar{v} \,dx\right| 
\le \varepsilon.
\end{equation}
Next we show that 
\begin{equation} \label{eq:2.16}
\sup_{v \in \Ha, \, \| v \|_{\Hal} \le 1}
\left| \int_{\{|x| \le R_{\eps}\}} \big\{ g(u_n) -g(u_0) \big\} \bar{v} \,dx\right| 
\to 0, \quad \text{as} \ n \to +\infty.
\end{equation}
First let us consider the case $N=3$.
Since $2<p<3$, we can take $q_0 \in \left( \frac{3}{2}, 2 \right]$ so that
\begin{equation} \label{eq:2.17}
1 \le q_0(p-1) <3 \quad \text{and} \quad q_0' \in [2,3),
\end{equation}
where $q_0'$ is the H\"older conjugate of $q_0$.
From \ef{eq:2.17}, it follows that
\begin{equation} \label{eq:2.18}
\|v\|_{q_0^{\prime}} \le C\|v\|_{\Hal} \quad \text { for all } \ v \in \Ha,
\end{equation}
\begin{equation} \label{eq:2.19}
\gl \in L^{q_0}(\mathbb{R}^3), \quad 
| \gl |^{p-1} \in L^{q_0}(\mathbb{R}^3),
\end{equation}
\begin{equation} \label{eq:2.20}
\phi_n \rightarrow \phi_0 \ \text{in} \ L_ {\rm loc}^{q_0}(\mathbb{R}^3) 
\quad \text {and} \quad 
|\phi_n|^{p-1} \rightarrow \left|\phi_0\right|^{p-1} \ \text{in} \ 
L_ {\rm loc}^{q_0}(\mathbb{R}^3).
\end{equation}
Moreover by \ef{eq:2.20}, we may assume that
\begin{equation} \label{eq:2.21}
\left|\phi_n\right| \le \Phi \quad \text { a.e. } \ x \in B_{R_{\varepsilon}}(0) \ 
\text{and} \ n \in \mathbb{N}
\end{equation}
for some $\Phi \in L^{q_0}_ {\rm loc}(\RT) \cap L^{(p-1)q_0}_ {\rm loc}(\R^3)$.
Then from \eqref{gCstima}, \ef{eq:2.19} and \ef{eq:2.21}, one has
\[
\begin{aligned}
\left| g(u_n)-g(u_0)\right|^{q_0} 
&\le C\left( |u_n|^{q_0}+|u_0|^{q_0}
+|u_n|^{(p-1) q_0}+|u_0|^{(p-1) q_0} \right) \\
&\le C\Big( |\phi_n|^{q_0}+|q(u_n)|^{q_0} \gl^{q_0}+|u_0|^{q_0} \\ 
&\qquad +|\phi_n|^{(p-1) q_0}+|q(u_n)|^{(p-1) q_0} \gl^{(p-1) q_0}
+|u_0|^{(p-1)q_0}\Big) \\
&\le C \Big( \Phi^{q_0}+M^{q_0} \gl^{q_0}+ |u_0|^{q_0} \\ 
&\qquad +\Phi^{(p-1) q_0}+M^{(p-1) q_0} \gl^{(p-1) q_0}
+|u_0 |^{(p-1)q_0} \Big) \in L^1_ {\rm loc}(\R^3),
\end{aligned}
\]
from which one finds that $g(u_n) \to g(u_0)$ in $L^{q_0}(\{|x| \le R_\eps\})$.
Thus from \ef{eq:2.18}, we obtain
\[
\begin{aligned}
&\sup_{ v \in \Ha, \, \| v \|_{\Hal} \le 1}
\left| \int_{\{|x| \le R_\eps\}} \big\{ g(u_n)-g(u_0) \big\} \bar{v} \,dx \right| \\
&\qquad \le \| g(u_n) -g(u_0) \|_{L^{q_0}(\{|x| \le R_{\eps}\})} 
\| v \|_{L^{q_0'}(\R^3)} \\
&\qquad \le C \| g(u_n)-g(u_0)\|_{L^{q_0}(\{|x| \le R_{\varepsilon})\}} 
\rightarrow 0 \ \text { as } \ n \rightarrow \infty,
\end{aligned}
\]
which shows that \ef{eq:2.16} holds.
In the case $N=2$, we have only to choose $q_0=2$.
From \ef{eq:2.15} and \ef{eq:2.16}, 
we arrive at \ef{eq:2.13}, which completes the proof.
\end{proof}

Now we analyse the relations between solutions of \eqref{eq2} and \eqref{eq}, 
boundary condition \eqref{bdry} and critical points of $I$.

\begin{proposition} \label{prop:2.8}
If $u$ is a solution of the original problem \ef{eq2}, 
then $u$ is a critical point of $I$.

On the other hand, if $u \in \Ha$ is a critical point of $I$,
then $u$ is a weak solution of \ef{eq}. 
If, in addition, $u=\phi_{\l}+q(u)\gl$ with $\phi_{\l}\in H^1(\RN)\cap C(\RN)$, 
then $u$ satisfies also the boundary condition \eqref{bdry}, up to a phase shift. 
Finally, if $u=\phi_{\l}+q(u)\gl$ with $\phi_{\l}\in H^2(\RN)$, 
then $u$ is a solution of \eqref{eq2}.
\end{proposition}

\begin{proof}
Although this fact has been shown in \cite{ABCT2}, 
we give the proof for the sake of completeness. 
Let  $u= \phi_\l + q(u) \gl \in D(-\Delta_\a)$ be a solution of \eqref{eq2}. 
Then, for any  $v= \psi_\l + q(v) \gl \in \Ha$, we have 
\[
\langle -\Delta \phi_\l - \l q(u) \gl -g(u), \psi_\l +q(v) \gl \rangle=0.
\]
In addition, by the definition of $\gl$, it follows that
\[
 \langle - \Delta \phi_\l+ \l \phi_\l, \gl \rangle
=\phi_{\lambda}(0)=(\alpha + \xi_\l ) q(u).
\]
Summing up and using \eqref{derivata}, we deduce that $I'(u)[v]=0$.

On the other hand, suppose that $I'(u)=0$.
Taking
$v= \psi_\l$ so that $q(v)=0$, we have
\[
0=I'(u)[\psi_\l]= \re \left\{ \langle \nabla \phi_\l, \nabla \psi_\l \rangle 
-\l q(u) \langle \gl, \psi_\l \rangle - \irn g(u) \overline{\psi_\l} \,dx \right\} 
\quad \text{for all} \ \psi_\l \in H^1(\RN), 
\]
from which we deduce that $\phi_\l$ is a weak solution of 
\begin{equation*} 
-\Delta \phi_\l - \l q(u) \gl = g(u) \quad \text{in} \ \R^N.
\end{equation*}

Suppose now that $u=\phi_{\l}+q(u)\gl$ is a critical point of $I$ 
with $\phi_{\l}\in H^2(\RN)$. 
Choosing $v=\gl$ so that $\psi_\l \equiv 0$ and $q(v)=1$, it follows that
\[
\re \left\{ -\l \langle \phi_\l + q(u) \gl, \gl \rangle
+(\alpha + \xi_\l) q(u) - \irn g(u) \gl \,dx \right\} 
=0.
\]
Using \ef{eq}, one finds that
\[
\re \phi_\l (0)
= \re \langle -\Delta \phi_\l + \l \phi_\l, \gl \rangle 
= \re \{ (\alpha + \xi_\l) q(u)\}.
\]
Hence, up to a phase shift, $u$ satisfies \ef{eq2}.
This completes the proof.
\end{proof}

\section{Existence of a nontrivial solution}\label{se:ex}

In this section, we establish the existence of a nontrivial solution of \ef{eq}
by applying the mountain pass theorem.

For this purpose, we set
\begin{equation*}
\Har := \{u\in \Ha:u \text{ radially symmetric}\}.
\end{equation*} 
Moreover, as in \cite{AP, BL, HIT, PW1, PW2}, 
we introduce an auxiliary nonlinear term as follows.
Let us fix $\omega_1\in (\omega_{\alpha},\omega)$,
where $\omega_\a $ and $\omega$ are respectively defined in 
\eqref{omega-al} and \eqref{g2},
and define 
\begin{equation} \label{hdef}
h(s):=
\max\{\omega_1 s+g(s),0\} \ \text{for} \ s\ge 0.
\end{equation}
We also extend $h$ to the complex plane similarly as $g$.
Then from \ef{g2}, we see that $h(s) \equiv 0$, for $|s| \sim 0$.
Thus by \ef{g3}, it holds that
\begin{equation} \label{eq:3.7}
\lim_{s \to 0} \frac{h(s)}{s} =0 \quad \text{and} \quad
\lim_{|s| \to +\infty} \frac{h(s)}{|s|^{p-1}}=0
\quad \text{for some }
\begin{cases}
2<p<3 &(N=3), \\
2<p<+\infty &(N=2).
\end{cases}
\end{equation}
Note that $p$ in \ef{eq:3.7} may be different to that of \ef{g3}.
From \ef{eq:3.7}, we also deduce that
for any $\eps>0$, there exists $C_{\eps}>0$ such that 
\begin{align}
h(s) &\le \eps s + C_{\eps} s^{p-1}, \quad \text{for} \ s \ge 0, \label{h<}.
\end{align}
Moreover from \ef{hdef}, it follows that 
\begin{align*}
g(s) &\le - \omega_1 s + h(s)\le -(\omega_1-\eps)s
+C_\eps s^{p-1},\quad \text{for} \ s \ge 0, 
\\
G(s) &\le - \frac{(\omega_1-\eps)}2 s^2+\frac{C_\eps}p s^{p}, 
\quad \text{for} \ s \ge 0. 
\end{align*}
Thus 
by definition of the extension to the complex plane of $g$ and $G$, 
we find that\begin{align}
g(u) \bar{u} = g(|u|)|u| 
&\le- \omega_1 |u|^2 + h(u) \bar{u}
\le -( \omega_1 - \eps) |u|^2 + C_{\eps} |u|^{p}, \quad \text{for} \ u \in \C, 
\label{g-est}\\
G(u) = G(|u|) &\le - \frac{\omega_1 - \eps}{2} |u|^2 + \frac{C_{\eps}}{p} |u|^p, 
\quad \text{for} \ u \in \C. \label{G<}
\end{align}

First we begin with the following.

\begin{lemma} \label{lem:MP}
Assume \eqref{g1}-\eqref{g4}.
Then the functional $I: \Ha \to \R$ has the mountain pass geometry, i.e.
\begin{enumerate}[label=(\roman{*}),ref=\roman{*}]
\setcounter{enumi}{0}
\item\label{MP1} there exist $\delta_0$, $\rho>0$ such that
$I(u) \ge \delta_0$ for $\|u\|_{\Hal}=\rho$;
\item\label{MP2} there exists $z \in \Har $ with $\|z \|_{\Hal}>\rho$ 
such that $I(z)<0$.
\end{enumerate}

\end{lemma}

\begin{proof}
\eqref{MP1}. \ 
Let $\l\in(\omega_\a, \omega_1)$
and $\eps \in (0, \omega_1 -\l)$, 
where $\omega_\a $ and $\omega_1$ are respectively defined in 
\eqref{omega-al} and \eqref{hdef}.
From \eqref{G<},
for any $u\in \Har$, we have
\begin{align*}
I(u) & \ge  \frac 12\|\n \phi_\l\|_2^2+\frac \l 2\|\phi_\l \|_2^2
+\frac{\omega_1 -\l -\eps}2\|u\|_2^2
+\frac 12(\a +\xi_\l)|q(u)|^2
 - \frac{C_\eps}{p} \|u\|_p^{p}.
\end{align*}
Then, by the Sobolev inequality, 
there exist $\delta_0$ and $\rho>0$ such that 
$I(u) \ge \delta_0$ for $\| u \|_{\Hal}= \rho$.

\eqref{MP2}. \
First we observe that when $u\in \Hr$, the set of radial functions of $\H$, 
it holds that $q(u)=0$ and 
\begin{equation*}
I(u)= \frac 12\|\n u\|_2^2
- \irn G(u) \,dx.
\end{equation*}
Then from \ef{g4} and following \cite{BL}, there exists $w\in \Hr$  such that 
\[
\irn G(w)\,dx>0.
\]
For any $t>0$, we set $w_t:=w(\cdot/t)$. Since
\[
I(w_t)=\frac{t^{N-2}}{2}\|\n w\|_2^2
-t^N\irn G(w)\,dx,
\]
for $t$ sufficiently large, we have that $I(w_t)<0$ 
with $\|w_t\|_{\Hal}=\|w_t\|_{H^1}>\rho$. 
This finishes the proof.
\end{proof}

By Lemma \ref{lem:MP}, denoting
\begin{equation*}
\Gamma := \left\{ \g \in C\big([0,1],\Har\big) \ : \ \g(0)=0,
I(\g(1))< 0\right\},
\end{equation*}
we infer that $\G$ is non-empty and
\begin{equation} \label{MP}
\sigma:=\inf_{\g \in \G }\max_{t\in[0,1]}
I(\g(t)) \ge \delta_0>0.
\end{equation}

Now, inspired by \cite{CT, HIT,jj}, 
we define the  functional $J:\R\times \Har\to \R$  as 
\begin{align} \label{eq:3.1}
\begin{split}
J(\theta, u) &:=\frac{e^{(N-2)\theta}}2 \|\n \phi_\l\|_2^2
+\frac{e^{(N-2) \theta}\l}{2}\big(\| \phi_\l\|_2^2
-\| u\|_2^2\big)
+\frac{e^{2(N-2) \theta}}{2} (\a +\xi_{e^{-2\t}\l})|q(u)|^2 
\\
&\qquad-e^{N\theta}\irn G(u) \,dx, 
\end{split}
\end{align}
for $u=\phi_\l+q(u) \cg_\l\in \Har$.
It is important to point out that $J(\theta, u) = I\big(u(e^{-\theta}\cdot )\big)$
as observed in Remark \ref{rem:2.6}.
Moreover by computing $\partial_{\theta}J(0,u)=0$,
we obtain the Pohozaev identity \ef{eq:5.6} formally.

With similar arguments of Lemma \ref{lem:MP}, 
$J$ also has the mountain pass geometry 
and we can define its mountain pass level as
\[
\tilde \sigma:=\inf_{(\theta,\g) \in \Sigma\times \G} 
\max_{t\in [0,1]}J\big(\theta(t),\g(t)\big), 
\]
where 
\[
\Sigma:=\left\{\theta \in C\big([0,1],\R\big) \ : \ \theta(0)=\theta(1)=0\right\}.
\]
Arguing as in \cite[Lemma 4.1]{HIT}, we derive the following.

\begin{lemma} \label{lem:MPJ}
The mountain pass levels of $I$ and $J$ coincide, namely $\s=\tilde \s$.
\end{lemma}

Now, as a immediate consequence of Ekeland's variational principle, 
we have the result below, 
whose proof can be found in \cite{CT}, \cite[Lemma 2.3]{jj}.

\begin{lemma} \label{lem:sqrte}
Let $\eps>0$. Suppose that $\eta \in \Sigma \times \G$ satisfies 
\[
\max_{t \in [0,1]}J( \eta(t))\le  \s+\eps.
\]
Then there exists $(\theta, u)\in \R\times \Har$ such that
\begin{enumerate}[label=(\roman{*}),ref=\roman{*}]
\setcounter{enumi}{0}
\item${\rm dist}_{\R \times \Har}\big((\theta,u),\eta([0,1])\big)
\le 2 \sqrt{\eps}$;
\item$J(\theta,u)\in [\s-\eps, \s+\eps]$;
\item $\|D J(\theta,u)\|_{\R \times (\Har)'}\le 2 \sqrt{\eps}$.
\end{enumerate}
\end{lemma}

Arguing as in \cite{CT} or \cite[Proposition 4.2]{HIT} 
and using Lemmas \ref{lem:MPJ} and \ref{lem:sqrte}, 
the following proposition holds.

\begin{proposition} \label{prop:ps-seq}
There exists a sequence $\{(\theta_n,u_n)\} \subset \R \times \Har$ such that, 
as $n \to +\infty$, 
\begin{enumerate}[label=(\roman{*}),ref=\roman{*}]
\setcounter{enumi}{0}
\item\label{ps-i} $\theta_n \to 0$;
\item \label{ps-ii}$J(\theta_n,u_n)\to \s$;
\item\label{ps-iii} $\de_\theta J(\theta_n,u_n)\to 0$;
\item\label{ps-iv} $\de_u J(\theta_n,u_n)\to 0$ \ strongly in $(\Har)'$. 
\end{enumerate}

\end{proposition}

Our next purpose is to establish the boundedness in $\Ha$ of the sequence 
$\{ u_n \}$ found in the previous lemma.

\begin{lemma}\label{pr:bdd}
Suppose that $N=3$, $\a>0$ and assume \ef{g1}-\ef{g4}.
Let $\{(\theta_n,u_n)\} \subset \R \times \Har$ 
be the sequence in Proposition \ref{prop:ps-seq}.
Then $\{u_n\}$ is bounded in $\Ha$.
\end{lemma}

\begin{proof}
We fix $\l\in (\omega_\a, \omega_1)$.
For any $n\ge 1$, we write $u_n=\phi_{\l,n}+q(u_n) \cg_\l$. 
For simplicity, we set $\phi_n:=\phi_{\l,n}$ and $q_n=q(u_n)$.
The proof is divided into two steps.

\smallskip
\noindent
\textbf{Step 1}. We show that $\| \nabla \phi_n \|_2$ and $\{ q_n \}$ are bounded.

Now by \eqref{ps-ii}-\eqref{ps-iii} of Proposition \ref{prop:ps-seq}, we have
\begin{align}
\begin{split}\label{uno}
&\frac{e^{(N-2)\theta_n}}2 \|\n \phi_n\|_2^2
+\frac{e^{(N-2) \theta_n}\l}{2}\big(\| \phi_n\|_2^2
-\| u_n\|_2^2 \big) \\
&\qquad+ \frac{e^{2(N-2) \theta_n}}{2} (\a +\xi_{e^{-2\t_n}\l})|q_n|^2 
-e^{N\theta_n}\irn G(u_n) \,dx=\s +o_n(1),
\end{split}
\\
\begin{split}\label{due}
&\frac{(N-2)e^{(N-2)\theta_n}}2  \|\n \phi_n\|_2^2
+\frac{(N-2)e^{(N-2) \theta_n}\l}{2}\big(\| \phi_n\|_2^2
-\| u_n\|_2^2\big) 
-e^{(N-2) \theta_n}\l\|\cg_\l\|_2^2|q_n|^2 \\
&\qquad+(N-2)e^{2(N-2) \theta_n} (\a +\xi_{e^{-2\t_n}\l})|q_n|^2 
-Ne^{N\theta_n}\irn G(u_n) \,dx=o_n(1).
\end{split}
\end{align}
Here we used the fact:
\[
\frac{d}{d\theta} \left( \xi_{e^{-2\theta} \l} \right) 
=\left\{\begin{array}{ll}
\frac{d}{d\theta} ( e^{-\theta} \xi_\l) = - e^{-\theta} \xi_\l &(N=3) \\[2mm]
\frac{d}{d\theta} (\red{\xi_\l}- \frac{\theta}{2\pi}) = - \frac{1}{2\pi} & (N=2) 
\end{array}\right\}
= - 2 e^{-(N-2) \theta} \l \| \gl \|_2^2.
\] 
Multiplying \eqref{uno} by $N$ and subtracting by \eqref{due} we deduce that
\begin{align}\label{unodue}
\begin{split}
N \s + o_n(1) &=
e^{(N-2)\theta_n}\|\n \phi_n\|_2^2
+e^{(N-2) \theta_n}\l\big(\| \phi_n\|_2^2
-\| u_n\|_2^2\big) \\
&\qquad +e^{(N-2) \theta_n}\l\|\cg_\l\|_2^2|q_n|^2
+\frac{4-N}{2} e^{2(N-2) \theta_n} (\a +\xi_{e^{-2\t_n}\l})|q_n|^2.
\end{split}
\end{align}
Note that unlike the regular case $q=0$ as \cite{HIT},
we are not able to conclude that $\| \nabla \phi_n \|_2$ is bounded
because of the second term.
To overcome this difficulty, we further distinguish into two cases.

\smallskip
\noindent
\textbf{Case 1}. Suppose that 
$\dis \liminf_{n \to +\infty} \big(\| u_n\|_2^2 -\| \phi_n\|_2^2\big) > 2$.

In this case, let us set
\begin{equation} \label{5.6-2}
\mu_n:=\frac{\l}{\| u_n \|_2^2 +\| \phi_n\|_2^2}.
\end{equation}
Then we have $0< \mu_n \le \frac{\l}{2}$.
It is also important to mention that possibly $\mu_n \to 0$.

Now we write $u_n = \psi_n + q_n \cg_{\mu_n}$.
Since the value of $I$ is independent of the choice of $\l$,
it follows that
\begin{align*}
\begin{split}
N \s + o_n(1) &=
e^{(N-2)\theta_n}\|\n \psi_n\|_2^2
+e^{(N-2) \theta_n} \mu_n \big(\| \psi_n\|_2^2 -\| u_n\|_2^2\big) \\
&\qquad +e^{(N-2) \theta_n} \mu_n \|\cg_{\mu_n} \|_2^2|q_n|^2
+\frac{4-N}{2} e^{2(N-2) \theta_n} (\a +\xi_{e^{-2\t_n} \mu_n})|q_n|^2.
\end{split}
\end{align*}
Moreover from \ef{5.6-2}, we have 
$\mu_n \| u_n \|_2^2 \le \l$ and hence
\begin{align} \label{5.6-3}
\begin{split}
N \s + \l e^{(N-2) \theta_n} + o_n(1)
&\ge e^{(N-2)\theta_n}\|\n \psi_n\|_2^2 
+e^{(N-2) \theta_n} \mu_n \|\cg_{\mu_n} \|_2^2|q_n|^2 \\
&\qquad +\frac{4-N}{2} e^{2(N-2) \theta_n} (\a +\xi_{e^{-2\t_n} \mu_n})|q_n|^2.
\end{split}
\end{align}
At this stage, let us suppose that $N=3$ and $\alpha>0$.
Then from \ef{xi} and Proposition \ref{prop:2.1}-\ef{2.1-5}, one has 
\[
\xi_{e^{-2\t_n} \mu_n} = e^{-\theta_n} \xi_{\mu_n} \quad \text{and}
\quad \mu_n \| \cg_{\mu_n} \|_2^2 = \frac{\xi_{\mu_n}}{2},
\]
from which we arrive at
\begin{equation} \label{5.6-4}
\red{e^{(N-2) \theta_n}\mu_n\|\cg_{\mu_n}\|_2^2 
+\frac{4-N}{2} e^{2(N-2) \theta_n} (\a +\xi_{e^{-2\t_n}\mu_n})
= e^{\theta_n} \xi_{\mu_n}  + \frac{\alpha}{2} e^{2 \theta_n} .}
\end{equation}
Since $\xi_{\mu_n}>0$ and $\alpha>0$, we deduce from \ef{5.6-3} that
\[
3 \sigma + \l e^{\theta_n} + o_n(1) 
\ge e^{\theta_n} \| \nabla \psi_n \|_2^2 + \frac{\alpha}{2} e^{2\theta_n} |q_n|^2,
\]
yielding that $\| \nabla \psi_n \|_2$ and $\{ q_n \}$ are bounded. 
Clearly this is not enough 
and we have to show that $\| \nabla \phi_n \|_2$ is bounded, too.

Recalling that $u_n= \phi_n + q_n \gl = \psi_n + q_n \cg_{\mu_n}$, we have
\begin{equation} \label{5.6-5}
\| \nabla \phi_n \|_2 \le \| \nabla \psi_n \|_2
+ |q_n| \| \nabla ( \cg_{\mu_n} - \gl) \|_2.
\end{equation}
By the Plancherel theorem, one also finds that
\[
\| \nabla ( \cg_{\mu_n} - \gl) \|_2
= \left\| |\xi| \left| \frac{1}{|\xi|^2 + \mu_n} - \frac{1}{|\xi|^2+\l} \right| \right\|_2
= | \l - \mu_n | \left\| \frac{|\xi|}{(|\xi|^2+ \mu_n)(|\xi|^2 + \l)} \right\|_2.
\]
Since $0< \mu_n \le \frac{\l}{2}$, it follows that
\begin{equation} \label{5.6-6}
\| \nabla ( \cg_{\mu_n} - \gl) \|_{L^2(\R^3)}
\le \l \left\| \frac{1}{|\xi|(|\xi|^2+\l)} \right\|_{L^2(\R^3)} < +\infty.
\end{equation}
Thus from \ef{5.6-5}, we conclude that $\| \nabla \phi_n \|_2$ is bounded.

\smallskip
\noindent
\textbf{Case 2}. Suppose that 
$\dis \liminf_{n \to +\infty} \big(\| u_n\|_2^2 -\| \phi_n\|_2^2\big) \le 2$.

In this case, passing to a subsequence, we may assume that
$\| u_n \|_2^2 - \| \phi_n \|_2^2 \le 3$.
Then from \ef{unodue}, one deduces that
\begin{align*}
N \sigma + 3 \l e^{(N-2)\theta_n} + o_n(1)
&\ge e^{(N-2)\theta_n} \| \nabla \phi_n \|_2^2 
+e^{(N-2)\theta_n} \l \| \gl \|_2^2 |q_n|^2 \\
&\qquad + \frac{4-N}{2}e^{2(N-2)\theta_n} (\alpha + \xi_{e^{-2\theta_n \l}}) |q_n|^2.
\end{align*}
Since, as $n \to +\infty$, $\alpha + \xi_{e^{-2\theta_n \l}} \to \alpha + \xi_\l >0$
for $\l \in (\omega_{\alpha}, \omega_1)$, 
we are able to obtain the boundedness of $\| \nabla \phi_n \|_2$ and $\{ q_n \}$.

\smallskip
\noindent
\textbf{Step 2}. We prove that $\| \phi_n \|$ is bounded.

Now by Proposition \ref{prop:ps-seq}-\eqref{ps-iv}, 
we know that $\|\de_uJ(\theta_n,u_n)\|_{(\Hal)'}=o_n(1)$ and thus
\begin{equation*}
\big| \de_uJ(\theta_n,u_n)[u] \big| =o_n(1)\|u\|_{\Hal}
\quad \text{for all }u\in \Har.
\end{equation*}
This implies that, if $u=\phi_\l+q(u) \cg_\l$, we have
\begin{align}
\begin{split}\label{unophi}
& \re \Big\{ e^{(N-2)\theta_n} \irn \n \phi_n \cdot \n \overline{\phi_\l} \, dx
+e^{(N-2) \theta_n}\l\irn \phi_n \overline{\phi_\l} \, dx
-e^{(N-2) \theta_n}\l\ird u_n \bar{u} \, dx  \\
&\qquad +e^{2(N-2)\theta_n} (\a +\xi_{e^{-2\theta_n} \l})q_n \overline{q(u)}  
-e^{N\theta_n}\irn g(u_n) \bar{u} \,dx \Big\} 
=o_n(1)\|u\|_{\Hal}.
\end{split}
\end{align} 
Suppose by contradiction that $\|\phi_n\|_2\to +\infty$.
Since $\| \phi_n \|_2 \le \| u_n \|_2 + |q_n| \| \gl \|_2 
\le \| u_n \|_2 +C$ by Step 1, it follows that $\| u_n \|_2 \to +\infty$ as well.
Let us put $t_n:=\| u_n\|_2^{-\frac 2N}\to 0$.
We also set
\[
v_n (x) := u_n \left(t_n^{-1} x \right)
=\phi_n \left(t_n^{-1} x \right) +q_n \cg_\l\left(t_n^{-1} x\right)
= \phi_n\left(t_n^{-1} x \right) + q_n t_n^{N-2} \cg_{\frac{\l}{t_n^2}} (x)
\]
and $\psi_n(x):=\phi_n\left(t_n^{-1}x\right)$.
Then one has
\[
\| v_n \|_2=1, \quad 
\| \psi_n \|_2^2 = t_n^N \| \phi_n \|_2^2
\le \frac{( \| u_n \|_2 +C)^2}{ \| u_n \|_2^2} \le C
\quad \text{and} \quad \| \nabla \psi_n \|_2^2 = t_n^{N-2} \| \nabla \phi_n \|_2^2 \to 0.
\]
Especially $\{\psi_n\}$ is bounded in $\H$.
Thus there exists $\psi_0\in \H$ such that, up to a subsequence, 
$\psi_n\weakto \psi_0$ weakly in $\H$ 
and $\psi_n \to \psi_0$ in $L^\tau_{\rm loc}(\RN)$ 
for $1\le \tau<6$ if $N=3$, and $1\le \tau<+\infty$ if $N=2$. 
Moreover, being $\{q_n\}$ bounded, 
there exists $q_0\in \C$ such that $q_n \to q_0$, up to a subsequence. 

Let $\vfi\in \H$ with compact support. 
Applying \eqref{unophi} to $u=\vfi(t_n \cdot)$, being $q(u)=0$, one finds that 
\begin{align*}
\begin{split}%\label{unovfi}
&\re \Big\{ e^{(N-2)\theta_n} t_n^{-(N-2)}\irn \n \psi_n \cdot \n \bar{\vfi} \, dx
-e^{(N-2) \theta_n}\l q_n
t_n^{-N} \irn \cg_{\l} (t_n^{-1}x) \bar{\vfi} \, dx \\
& \qquad-e^{N\theta_n}t_n^{-N}\irn g(v_n) \bar{\vfi} \,dx \Big\} 
=o_n(1)\sqrt{ t_n^{-(N-2)}\|\n \vfi\|_2^2+t_n^{-(N-2)} \l \|\vfi\|_2^2}.
\end{split}
\end{align*} 
Multiplying by $t_n^N$, we obtain
\begin{align}
\begin{split}\label{unovfi}
& \re \Big\{ e^{(N-2)\theta_n} t_n^{2}\irn \n \psi_n \cdot \n \bar{\vfi} \, dx
-e^{(N-2) \theta_n}\l q_n \irn \cg_\l\left(t_n^{-1}x\right) \bar{\vfi} \, dx \\
&\qquad-e^{N\theta_n}\irn g(v_n) \bar{\vfi} \,dx \Big\} 
=o_n(1)t_n^{\frac{N+2}{2}}\sqrt{ \|\n \vfi\|_2^2+ \l \|\vfi\|_2^2}.
\end{split}
\end{align} 
Since $\cg_\l$ is radially decreasing and decays to zero at infinity, we have
\[
\left| \cg_\l\left(t_n^{-1}x\right) \overline{\vfi(x)} \right| \to 0 
\quad \text{for all }x\neq 0 \ \text{ as }n \to +\infty.
\]
Moreover, for sufficiently large $n \in \N$, it holds that
\[
|\cg_\l\left(t_n^{-1}x\right) \overline{\vfi(x)} |
=|\cg_\l\left( |t_n^{-1}x| \right) \overline{\vfi(x)} |
\le |\cg_\l(|x|) \overline{\vfi(x)} |\in L^1(\RN)
\]
and by the Lebesgue dominated convergence theorem, we find that
\begin{equation*}
\irn \cg_\l\left(t_n^{-1}x\right) \overline{\vfi(x)} \, dx \to 0 
\quad \text{ as }n \to +\infty.
\end{equation*}
Moreover by Proposition \ref{prop:2.1}-\eqref{2.1-2}, one has
\begin{equation*}%\label{glL}
\| \cg_\l\left(t_n^{-1}\cdot\right)\|_\tau^\tau
=t_n^N\| \cg_\l\|_\tau^\tau \to 0
\quad \text{ for }
\begin{cases}
1\le \tau<3 &(N=3), \\
1 \le \tau<+\infty &(N=2),
\end{cases}
\end{equation*} 
and hence we deduce that $v_n \to \psi_0$ a.e. in $\R^N$ and
\begin{equation*}%\label{glL}
v_n\to \psi_0
\quad \text{in} \ L^\tau_{\rm loc}(\RN) \quad \text{for} \ 
\begin{cases}
1\le \tau<3 &(N=3), \\
1\le \tau<+\infty &(N=2).
\end{cases}
\end{equation*} 
Thus we obtain 
\[
\irn g(v_n) \bar{\vfi} \,dx \to \irn g(\psi_0) \bar{\vfi} \,dx 
\quad \text{ as }n \to +\infty.
\]
Hence, by \eqref{unovfi}, we infer that
\[
\re\irn g(\psi_0) \bar{\vfi} \,dx=0 
\quad\text{for all $\vfi\in \H$ with compact support}.
\]
This implies that $g(\psi_0)=0$ and, thanks to \eqref{g2}, then $\psi_0 \equiv 0$. 

Next since
\begin{equation*}
\big| \de_uJ(\theta_n,u_n)[u_n] \big| =o_n(1)\|u_n\|_{\Hal},
\end{equation*}
we have
\begin{align*}
&e^{(N-2)\theta_n} \| \n \phi_n \|_2^2
+e^{(N-2) \theta_n}\l \left( \|\phi_n \|_2^2 - \|u_n\|_2^2 \right)
+(\a +\xi_{e^{-2\theta_n}\l}) e^{2(N-2)\theta_n} |q_n|^2 \\
&-e^{N\theta_n}\irn g(u_n) \overline{u_n} \,dx =o_n(1)\|u_n\|_{\Hal}.
\end{align*}
Thus one finds that
\begin{align*}
& e^{(N-2)\theta_n}t_n^{-(N-2)} \| \n \psi_n \|_2^2
+e^{(N-2) \theta_n}t_n^{-(N-2)} \l \left( \|\psi_n \|_2^2 - \|v_n\|_2^2 \right) \\
&\qquad 
+ (\a +\xi_{e^{-2\theta_n} t_n^2 \l}) e^{2(N-2)\theta_n} t_n^{-2(N-2)} |q_n|^2
-e^{N\theta_n} t_n^{-N}\irn g(v_n) \overline{v_n} \,dx \\
&=o_n(1)\sqrt{ t_n^{-(N-2)}\|\n \psi_n\|_2^2+t_n^{-(N-2)}\l \|\psi_n\|_2^2
+ t_n^{-2(N-2)} |q_n|^2}.
\end{align*} 
Multiplying by $t_n^N$, we obtain
\begin{align*}
& e^{(N-2)\theta_n}t_n^{2} \| \n \psi_n \|_2^2
+e^{(N-2) \theta_n} t_n^2 \l \left( \|\psi_n \|_2^2 - \|v_n\|_2^2 \right) 
+ (\a +\xi_{e^{-2\theta_n} t_n^2 \l}) e^{2(N-2)\theta_n} t_n^{4-N} |q_n|^2 \\
&\qquad =e^{N\theta_n}\irn g(v_n) \overline{v_n} \,dx 
+o_n(1)t_n^2
\sqrt{ t_n^{N-2} \|\n \psi_n\|_2^2+ t_n^{N-2} \l \|\psi_n\|_2^2+ |q_n|^2}.
\end{align*} 
Now for $\l \in (\omega_{\alpha}, \omega_1)$, 
we have from $\alpha>0$, $\xi_\l >0$, $t_n \to 0$ 
and $e^{(N-2)\theta_n} t_n^2 \le e^{\theta_n}$ that
\begin{align*}
e^{(N-2)\theta_n} t_n^2 \l \| \psi_n \|_2^2 
+ e^{N \theta_n}( \omega_1 - \l) \| v_n \|_2^2 
&\le e^{(N-2)\theta_n} t_n^2 \l \| \psi_n \|_2^2 
+ \left( e^{N \theta_n} \omega_1 - e^{(N-2)\theta_n} t_n^2 \l \right) \| v_n \|_2^2 \\ 
&\le e^{(N-2)\theta_n} t_n^{2} \| \nabla \psi_n \|_2^2
+e^{(N-2)\theta_n} t_n^2 \l \left( \| \psi_n\|_2^2 - \| v_n \|_2^2 \right) \\
&\quad + e^{N \theta_n} \omega_1 \| v_n \|_2^2
+( \alpha \xi_{e^{-2\theta_n} t_n^2 \l}) e^{2(N-2)\theta_n} t_n^{4-N} |q_n|^2.
\end{align*}
In addition, by \ef{g-est}, one also finds that
\[
e^{N \theta_n} \irn g(v_n) \overline{v_n} \,dx
\le e^{N\theta_n} \irn h(v_n) \overline{v_n} \,dx
- e^{N\theta_n} \omega_1 \| v_n \|_2^2.
\]
Thus, from $\| v_n \|_2=1$ and the last three inequalities deduce that
\begin{align} \label{eq:3.6}
e^{(N-2) \theta_n} (\omega_1 - \l) 
&\le e^{N\theta_n}\irn h (v_n) \overline{v_n} \,dx
+o_n(1)t_n^2 
\sqrt{ t_n^{N-2} \|\n \psi_n\|_2^2+ t_n^{N-2} \l \|\psi_n\|_2^2+ |q_n|^2}.
\end{align} 
On the other hand by Radial Strauss Lemma \cite{Str}, there exists $C>0$ such that, 
for any $n\ge1$ and $x\in \RN$ with $|x|\ge 1$, it holds that
\[
\begin{aligned}
|v_n(x)| &\le |\psi_n(  x)|+|q_n| |\cg_\l( t_n^{-1} x)|
&\le \frac{C}{|x|^{\frac{N-1}{2}}} \| \psi_n \|_{H^1} 
+ |q_n| |\cg_\l(|x|)| 
\le \frac{C}{|x|^{\frac{N-1}2}}.
\end{aligned}
\] 
Then from \ef{eq:3.7},
we are able to apply the Strauss compactness lemma
\cite[Theorem A.1]{BL}, \cite[Lemma 2]{Str} to obtain
\[
\irn h(v_n) \overline{v_n} \,dx \to \irn h(\psi_0) \overline{\psi_0} \,dx =0
\quad \text{as} \ n \to +\infty.
\]
Thus from \ef{eq:3.6}, we deduce that $0< \omega_1 - \l \le 0$,
reaching a contradiction 
and proving the boundedness $\{\phi_n\}$ in $L^2(\RN)$, as desired.
\end{proof}

\begin{remark} \label{rem:5.7}
When $N=3$ and $\alpha \red{\le} 0$, the argument of Step 1 of Lemma \ref{pr:bdd} fails.
Indeed since $\mu_n$ may goes to $0$, 
we cannot see if the right hand side of \ef{5.6-4} is positive when $\alpha \red{\le} 0$.

In the case $N=2$, we can also observe from \ef{xi}
and Proposition \ref{prop:2.1}-\ef{2.1-5} that
\[
\begin{aligned}
&\frac{4-N}{2} (\alpha + \xi_{e^{-2\theta_n} \mu_n } ) e^{2(N-2) \theta_n} |q_n|^2
+ e^{(N-2) \theta_n} \mu_n \| \cg_{\mu_n} \|_2^2 |q_n|^2 \\
&= \alpha |q_n|^2 
+ \left( \xi_\l - \frac{1}{4\pi} \log ( \| u_n \|_2^2 + \| \phi_n \|_2^2)
- \frac{\theta_n}{2\pi} \right)|q_n|^2 + \frac{1}{4\pi} |q_n|^2.
\end{aligned}
\]
But since $\| u_n \|_2$ may go to $+\infty$, we cannot conclude
the boundedness of $\{ q_n \}$, as before.
Moreover when $N=2$, it follows that
\[
\left\| \frac{1}{|\xi| (|\xi|^2+\l)} \right\|_{L^2(\R^2)}= +\infty,
\]
implying that the boundedness of $\| \nabla \phi_n \|_2$
from that of $\| \nabla \psi_n \|_2$ is unclear.

It is also worth mentioning that Step 2 of Lemma \ref{pr:bdd} works well
even if $N=3$, $\alpha \red{\le} 0$ or $N=2$.
\end{remark} 

As explained in the previous remark, whenever $N=3$, $\alpha \red{\le} 0$ or $N=2$, 
the previous arguments do not work under the assumptions \ef{g1}-\ef{g4}. 
Therefore, in this case, in place of \eqref{g4}, we have to require \eqref{g5}. 
Observe that, under this growth condition, the situation is more straightforward. 
In particular,  the auxiliary functional $J$ is no more necessary 
and we can directly deal with classical Palais-Smale sequences.

\begin{lemma} \label{lem:5.7}
Suppose that $N=3$, $\alpha \red{\le} 0$ or $N=2$.
Assume \ef{g1}-\ef{g3} and \eqref{g5}.
For $c>0$, let $\{ u_n \} \subset \Ha$ be a $(PS)_c$-sequence for $I$.
Then $\{ u_n \}$ is bounded in $\Ha$.
\end{lemma}

\begin{proof}
We fix $\l \in (\omega_{\alpha}, \omega)$ and decompose $u_n = \phi_n + q(u_n) \gl$. 
Then one has
\begin{align*}
c+o_n(1) &= \frac{1}{2} \| \nabla \phi_n \|_2^2
+ \frac{\l}{2} \| \phi_n \|_2^2
+ \frac{\omega - \l}{2} \| u_n \|_2^2
+ \frac{\alpha + \xi_\l}{2} |q(u_n)|^2
- \irn H(u_n) \,dx, \\
o_n(1) \| u_n \|_{\Hal}
&= \| \nabla \phi_n \|_2^2 + \l \| \phi_n \|_2^2 + (\omega - \l) \| \phi_n \|_2^2
+(\alpha + \xi_\l) |q(u_n)|^2 - \irn h(u_n) \overline{u_n} \,dx.
\end{align*}
Thus from \eqref{g5}, we deduce that
\[
\beta c - o_n(1) \| u_n \|_{\Hal}
\ge \frac{\beta -2}{2} 
\left( \| \nabla \phi_n \|_2^2 + \l \| \phi_n \|_2^2
+ (\omega - \l) \| u_n \|_2^2 + (\alpha + \xi_\l) |q(u_n)|^2 \right),
\]
yielding that $\| u_n \|_{\Hal}$ is bounded.
\end{proof}

\begin{proposition} \label{pr:3.7}
Assume \ef{g1}-\ef{g4} if $N=3$, $\alpha>0$,
and \ef{g1}-\ef{g3} and \eqref{g5} if $N=3$, $\alpha \red{\le} 0$ or $N=2$.
Then $I$ has a nontrivial critical point $u_0 \in \Ha$ of mountain pass type, 
namely $I(u_0) = \sigma$, where $\s$ is defined in \eqref{MP}.
\end{proposition}

\begin{proof}
First we consider the case $N=3$ and $\alpha>0$.
Let $\{u_n\} \subset \Har$ be the sequence in Proposition \ref{prop:ps-seq} 
and fix $\l\in(\omega_\a,\omega)$.
For any $n\ge 1$, we decompose $u_n=\phi_{\l,n}+q(u_n) \cg_\l$. 
For simplicity, we set $\phi_n:=\phi_{\l,n}$ and $q_n=q(u_n)$.
By Lemma \ref{pr:bdd}, we know that $\{ u_n \}$ is bounded in $\Ha$. 
Especially $\| \phi_n \|_{H^1(\R^N)}$ and $\{ q_n \}$ are bounded. 
Thus there exist $\phi_0\in \H$ and $q_0\in \C$ such that, up to subsequences, 
$\n \phi_n \rightharpoonup \n \phi_0$ weakly in $L^2(\RN)$,
$\phi_n \rightharpoonup \phi_0$ weakly in $L^2(\RN)$ 
and almost everywhere in $\RN$, and $q_n \to q_0$ as $n \to +\infty$. 
We set $u_0:=\phi_0+q_0\cg_\l$.
Clearly $u_0$ is a critical point of $I$ and we aim to prove that it is nontrivial.

Now by Radial Strauss Lemma \cite{Str}, there exists $C>0$ such that, 
for any $n\ge1$ and $x\in \RN$ with $|x|\ge 1$, it holds that
\[
|u_n(x)|\le |\phi_n(x)|+|q_n|\cg_\l(x)\le \frac{C}{|x|^{\frac{N-1}2}}.
\] 
Then from \ef{eq:3.7}, we can apply the Strauss compactness lemma again
to deduce that
\begin{equation*}
\irn h(u_n) \overline{u_n} \,dx\to \irn h(u_0) \overline{u_0} \,dx
\quad \text{ as }n\to +\infty. 
\end{equation*}

Next by \eqref{ps-iv} of Proposition \ref{prop:ps-seq}, 
for any $u=\phi_\l+q(u) \cg_\l\in \Har$, we have
\begin{align*}
& \re \Big\{ e^{(N-2)\theta_n} \irn \n \phi_n \cdot \n \overline{\phi_\l} \, dx
+e^{(N-2) \theta_n}\l\irn \phi_n \overline{\phi_\l} \, dx
-e^{(N-2) \theta_n}\l\ird u_n \bar{u} \, dx  \\
&\qquad + (\a +\xi_{e^{-2\theta_n}\l}) e^{2(N-2)\theta_n} q_n \overline{q(u)} 
-e^{N\theta_n}\irn g(u_n) \bar{u} \,dx \Big\} 
=o_n(1),
\end{align*} 
and hence
\[
\re \left\{ \irn \n \phi_0 \cdot \n \overline{\phi_\l} \, dx
+\l\irn \phi_0 \overline{\phi_\l} \, dx
-\l\ird u_0 \bar{u} \, dx
+(\a +\xi_\l) q_0 \overline{q(u)} 
-\irn g(u_0) \bar{u} \,dx \right\} 
=0.
\]
In particular, we have
\begin{equation*}
 \|\n \phi_0 \|_2^2
+\l\| \phi_0 \|_2^2
-\l\| u_0 \|_2^2
+(\a +\xi_\l)|q_0|^2
-\irn g(u_0) \overline{u_0} \,dx=0.
\end{equation*}
Again by \eqref{ps-iv} of Proposition \ref{prop:ps-seq}, we also have
\begin{align*}
& e^{(N-2)\theta_n} \|\n \phi_n \|_2^2
+e^{(N-2)\theta_n}\l\| \phi_n \|_2^2
-e^{(N-2)\theta_n}\l\| u_n \|_2^2 \\
&\qquad +(\a +\xi_{e^{-2\theta_n}\l}) e^{2(N-2)\theta_n} |q_n|^2
-e^{N\theta_n}\irn g(u_n) \overline{u_n} \,dx=o_n(1).
\end{align*}
Therefore, since $e^{(N-2)\theta_n} -e^{N\theta_n}\to 0$,
$\alpha \xi_{e^{-2\theta_n} \l} \to \alpha + \xi_\l >0$ and
\begin{align*}
&e^{(N-2)\theta_n} \left( \|\n \phi_n \|_2^2
+ \l \| \phi_n \|_2^2 +(\omega_1-\l) \| u_n \|_2^2
+ (\a +\xi_\l)|q_n|^2 \right) \\
&\quad= 
e^{N\theta_n} \omega_1 \| u_n \|_2^2
+\left(e^{(N-2)\theta_n} -e^{N\theta_n} \right)\omega_1 \| u_n \|_2^2
+(\alpha + \xi_{e^{-2\theta_n}\l}) ( e^{(N-2)\theta_n} - e^{2(N-2)\theta_n}) |q_n|^2 \\
&\qquad +(\xi_\l - \xi_{e^{-2\theta_n}\l})e^{(N-2)\theta_n} |q_n|^2
+e^{N\theta_n}\irn g(u_n) \overline{u_n} \,dx+o_n(1) \\
&\quad = e^{N\theta_n} \irn h(u_n) \overline{u_n} \,dx
-e^{N\theta_n}\irn \!\left(h(u_n) \overline{u_n} - \omega_1 |u_n|^2
-g(u_n) \overline{u_n} \right)dx+o_n(1),
\end{align*} 
arguing in \cite{HIT}, we deduce that
\begin{multline*}
\limsup_{n \to \infty} \left( \|\n \phi_n \|_2^2
+\l\| \phi_n \|_2^2
+(\omega_1-\l)\| u_n \|_2^2
+(\a +\xi_\l)|q_n|^2\right) \\
\le \|\n \phi_0 \|_2^2
+\l\| \phi_0 \|_2^2
+(\omega_1-\l)\| u_0 \|_2^2
+(\a +\xi_\l)|q_0|^2
\end{multline*}
and, by the weak lower semi-continuity of the norm, 
we conclude that $u_n \to u_0$ strongly in $\Ha$.
Then by \eqref{ps-ii} of Proposition \ref{prop:ps-seq}, 
we deduce that $I(u_0)=\s$ and hence $u_0$ is nontrivial.

When $N=3$, $\alpha \red{\le} 0$ or $N=2$, 
we know that any $(PS)_{\sigma}$-sequence is bounded by Lemma \ref{lem:5.7}.
Then working on $I'(u_n)[u_n]$, we arrive at $u_n \to u_0$ in $\Ha$ and $I(u_0)=\sigma$.
This completes the proof. 
\end{proof}

Next we aim to prove that $q(u_0) \ne 0$
for the nontrivial solution $u_0$ obtained in Proposition \ref{pr:3.7},
which implies that our solution is actually singular.
For this purpose, let us recall some facts for the scalar field equation
\begin{equation} \label{eq:4.7}
- \Delta u =g(u) \qquad \text{in} \ \R^N
\end{equation}
in the complex-valued setting. 
To clarify the difference with \ef{eq}, 
let us write the energy functional $I_0$ 
associated with \ef{eq:4.7} as
\[
I_0(u) = \frac{1}{2} \| \nabla u \|_2^2 - \irn G(u) \,dx, 
\qquad \text{for } u \in H^1(\RN,\C).
\]
We also denote by $m_0$ the ground state energy level for $I_0$, namely,
\[
m_0 := \inf \left\{ I_0(u) \,:\, u \in H^1(\R^N,\C) \setminus \{ 0 \}, I_0'(u)=0 \right\}.
\]
If $m_0$ is achieved by some $u \in H^1(\RN,\C)$,
$u$ is said to be a \textit{ground state solution} of \ef{eq:4.7}.

Then we have the following.

\begin{lemma} \label{lem:4.8}
Assume \ef{g1}-\ef{g4}. Then the following hold:

\begin{enumerate} [label=(\roman{*}),ref=\roman{*}]

\item there exists $w \in H^1(\RN:\C)$ such that 
$I_0(w)=m_0$ and $I'_0(w)=0$;

\item\label{4.8-2} any ground state solution of \ef{eq:4.7} is
real-valued and positive on $\RN$, up to phase shift;

\item\label{4.8-3} there exists $\gamma_0 \in \Gamma_{0,{\rm real}}$ such that
\begin{equation} \label{eq:4.8}
\max_{t \in [0,1]} I_0 \big( \gamma_0(t) \big) = m_0,
\end{equation}
where
$\dis 
\Gamma_{0,{\rm real}} := \left\{ \gamma \in C \big( [0,1], H^1(\R^N,\R) \big) \,:
\,
\gamma(0)=0, I_0 \big( \gamma(1) \big) <0 \right\}$.
\end{enumerate}

\end{lemma}

\begin{proof}
First, by the classical result due to \cite{BL}, 
there exists a ground state solution $w \in H^1(\R^N,\R)$.
Moreover by the variation characterization of the ground state energy level
established in \cite{JT},
arguing as in \cite{AW, CiJSe, CJS},
we see that if $u$ is a ground state solution, 
then $|u|$ is also a ground state solution.
Then we are able to show that any ground state solution of \ef{eq:4.7}
has the form $u(x)= e^{i\theta} |u(x)|$ for $\theta \in \R$
and the positivity follows by the maximum principle.

Finally, \eqref{4.8-3} is a direct consequence of the result in \cite[Theorem 0.2]{JT}.
\end{proof}

Using Lemma \ref{lem:4.8}, we are able to prove the following.

\begin{proposition} \label{prop:4.9}
Let $u_0 \in \Ha$ be the nontrivial critical point of $I$ obtained in 
Proposition \ref{pr:3.7}. 
Then it holds that $q(u_0) \ne 0$.
\end{proposition}

\begin{proof}
First we claim that
\begin{equation} \label{eq:4.10}
\sigma \le m_0,
\end{equation}
where $\sigma$ is the mountain pass value of $I$ defined in \ef{MP}.
In fact, since $I(u) = I_0(u)$ for all $u \in H^1(\R^N,\C)$
and $\gamma_0 \in \Gamma_{0,{\rm real}} \subset \Gamma$, 
one finds from \ef{eq:4.8} that
\[
\sigma = \inf_{\gamma \in \Gamma} \max_{t \in [0,1]} I\big( \gamma(t) \big)
\le \max_{t \in [0,1]} I_0 \big( \gamma_0(t) \big) = m_0.
\]

Now by Proposition \ref{pr:3.7}, 
we know that $I(u_0) = \sigma$ and $I'(u_0)=0$. 
If $q(u_0)=0$, it follows that $u_0= \phi_\l \in H^1(\R^N,\C) \setminus \{0 \}$, 
yielding that
\[
I(u_0) = I_0(\phi_\l) \quad \text{and} \quad
I'(u_0)|_{H^1(\RN)} = I_0'(\phi_\l).
\]
This implies that
\[
\sigma = I(u_0) = I_0(\phi_\l) \ge m_0.
\]
Thus from \ef{eq:4.10}, we find that
\[
I_0(\phi_\l) = m_0 \quad \text{and} \quad I_0'(\phi_\l) =0,
\]
namely, $\phi_\l$ is a ground state solution of \ef{eq:4.7}.
Then by Lemma \ref{lem:4.8}-\eqref{4.8-2}, 
it holds that $\phi_\l$ is real-valued and positive on $\RN$, up to phase shift.

On the other hand, since $u_0$ is a weak solution of \ef{eq},
we can see by Proposition \ref{prop:5.1} that 
$\phi_{\l}\in H^1(\RN)\cap C(\RN)$ and so, by Proposition \ref{prop:2.8},
$u_0$ satisfies the boundary condition \ef{bdry}.
But if $q(u_0)=0$, \ef{bdry} shows that $\phi_\l(0)=0$,
contradicting to the positivity of $\phi_\l$.
Thus we conclude that $q(u_0) \ne 0$, as claimed.
\end{proof}

\begin{remark} \label{rem:4.10}
Proposition \ref{prop:4.9} heavily relies on the variational characterization of $u_0$.
It is not clear whether there exists another nontrivial solution $u$ of \ef{eq}
with $q(u) \ne 0$.
Moreover since we don't know if the mountain pass solution $u_0$ is 
a ground state solution for \ef{eq}, we cannot conclude that 
the strict inequality in \ef{eq:4.10} holds, 
which was performed in \cite{ABCT2, ABCT3} 
for the case $g(s)=-\omega s +|s|^{p-2}s$.  
\end{remark}

Now Theorems \ref{main} and \ref{thm:1.4} follow by 
Propositions \ref{prop:2.8}, \ref{pr:3.7} and \ref{prop:4.9}.

\subsection*{Acknowledgment}
The authors are grateful to the anonymous referees for 
carefully reading the manuscript and providing us valuable comments.

The first author is partly financed by European Union - Next Generation EU - 
PRIN 2022 PNRR ``P2022YFAJH Linear and Nonlinear PDE's: 
New directions and Applications", by INdAM - GNAMPA Project 2024 
''Metodi variazionali e topologici per alcune equazioni di Schrodinger nonlineari''
and by the Italian Ministry of University and Research under the Programme 
''Department of Excellence'' Legge 232/2016 (Grant No. CUP - D93C23000100001).
The second author is supported by 
JSPS KAKENHI Grant Numbers 
JP21K03317, JP24K06804.

This work has been partially carried out during a stay A.P. in Kyoto. 
He would like to express his deep gratitude to the 
Department of Mathematics of the Faculty of Science, Kyoto Sangyo University, 
for the support and warm hospitality.

\end{document}